\documentclass[11pt,oneside]{amsart}
 \usepackage[utf8]{inputenc}
 \usepackage[T1]{fontenc}
 \usepackage[english]{babel}
 \usepackage{csquotes}
 \usepackage[unicode=true,pdfusetitle, bookmarks=true,bookmarksnumbered=false,bookmarksopen=true,bookmarksopenlevel=3, breaklinks=false,pdfborder={0 0 0},backref=false,colorlinks=true]{hyperref}
 \hypersetup{linkcolor=blue,anchorcolor=blue,citecolor=blue,urlcolor=black} 
 \usepackage{amsfonts}
 \usepackage{amssymb}
 \usepackage{amsmath}
 \usepackage{color}
 \usepackage{dsfont}
 \usepackage{amsthm}
 \usepackage{ytableau}
 \usepackage{enumitem}
 \usepackage{mathtools}
 \usepackage{a4wide}
 \usepackage{mathrsfs}
 \usepackage{appendix}
 \usepackage{bm}

 \newtheorem{thm}{Theorem}[section]
 \newtheorem{theorem}[thm]{Theorem}
 
 \newtheorem{proposition}[thm]{Proposition}
 \newtheorem{lemma}[thm]{Lemma}
 \newtheorem{corollary}[thm]{Corollary}

 \newtheorem*{theorem*}{Theorem}

 \theoremstyle{definition}
 \newtheorem{definition}[thm]{Definition}

 \newtheorem{remark}[thm]{Remark}
 \numberwithin{equation}{section} 

\newcommand{\bC}{{\mathbb{C}}}

\newcommand{\bE}{{\mathbb{E}}}

\newcommand{\bN}{{\mathbb{N}}}

\newcommand{\bR}{{\mathbb{R}}}

\newcommand{\bZ}{{\mathbb{Z}}}

\newcommand{\cA}{{\mathcal{A}}}
\newcommand{\cB}{{\mathcal{B}}}

\newcommand{\cD}{{\mathcal{D}}}

\newcommand{\cM}{{\mathcal{M}}}
\newcommand{\cN}{{\mathcal{N}}}

\newcommand{\cP}{{\mathcal{P}}}
\newcommand{\cQ}{{\mathcal{Q}}}

\newcommand{\dsone}{\mathds{1}}

\DeclareMathOperator{\supp}{supp}

\begin{document}

\title[Noncommutative convolution operators]{Noncommutative maximal weak type $(1,1)$ estimates for convolution operators under $L_1$-H\"ormander conditions}

\author{Guixiang Hong}
\address{Guixiang Hong, Institute for Advanced Study in Mathematics, Harbin Institute of Technology, 15000 Harbin, China}
\email{gxhong@hit.edu.cn}

\author{Liang Wang}
\address{Liang Wang, Department of Mathematics, City University of Hong Kong, Hong Kong SAR, China}
\email{L.Wang@cityu.edu.hk}

\author{Xumin Wang}
\address{Xumin Wang, Institute for Advanced Study in Mathematics, Harbin Institute of Technology, 15000 Harbin, China}
\email{xumin.wang@hit.edu.cn}

\author{Bang Xu}
\address{Bang Xu, School of Mathematical Sciences, Xiamen University, Xiamen 361005, China}
\email{bangxu@xmu.edu.cn}

\begin{abstract}
   We prove weak type $(1,1)$ estimates for families of scalar convolution operators acting on operator-valued
function spaces. A joint $L_1$-H\"ormander condition, together with a maximal strong type estimate at one exponent $1<r\leq\infty$, suffices. The kernels may be complex-valued and need not have uniformly bounded $L_1$ norms. Under a size condition, the result also yields a weak endpoint estimate for all sharp truncation radii of singular kernels. This resolves, for convolution kernels, the problem mentioned in \cite[Remark~5.5]{parcet09pseudo} and recalled in Parcet's recent survey \cite[Remark~1.3]{Parcet2026ICM}. The proof involves not only a  simultaneous dyadic representation of the operators but also a noncommutative maximal weak type $(1,1)$ estimates of the resulting dyadic operators. This approach seems entirely new even in the commutative setting.
    
 Numerous applications are provided. On the one hand, it yields various maximal inequalities and pointwise convergence results for integrable kernels and dyadic approximations of the identity, monotone $L\log L$ majorants, and Poisson integrals on nilpotent groups and symmetric spaces. On the other hand, via various transference techniques, it implies maximal inequalities for Fourier multipliers, Schur multipliers and double operator integrals, together with
bilateral almost uniform convergence of dyadic Fourier summation methods on $L_1$ of group von Neumann algebras.
\end{abstract}

\maketitle
\pagestyle{headings}

\tableofcontents

\section{Introduction and main results}\label{sec:introduction}

Maximal inequalities for convolution operators provide
a means of passing from convergence on a dense class
to almost everywhere convergence.
For integrable approximation kernels,
Z\'o~\cite{zo76appro} proved weak type $(1,1)$ and
strong type $(p,p)$ estimates under a joint
$L_1$-H\"ormander condition.
Z\'o's result turned out to be a standard application
of vector-valued Calder\'on--Zygmund theory, having
found numerous applications including a new proof of Hardy-Littlewood maximal inequality; see, e.g.,
\cite[Chapter~II, Section~4]{stein93bookharmonic}.

We study these maximal inequalities for operator-valued
functions.
Let $\cM$ be a semifinite von Neumann algebra equipped
with a normal faithful semifinite trace $\tau$, and put
$$
\cN=L_\infty(\bR^d)\overline\otimes\cM,\qquad
\varphi=\int_{\bR^d}dx\otimes\tau.
$$
In this setting, the noncommutative Hardy--Littlewood
maximal inequality was established by
Mei~\cite{mei07hardy}.
Mei's approach was built on previous foundational works
on noncommutative martingale and ergodic inequalities
and vector-valued noncommutative $L_p$-spaces
\cite{cuculescu71martingales,pisier98noncommutativeLp,
pisierxu97martingale,junge02doob,jungexu07ergodic};
see also \cite{lance76ergodic,yeadon77ergodic}
for earlier contributions.
The first noncommutative Calder\'on--Zygmund
decomposition was obtained by Parcet~\cite{parcet09pseudo},
who exploited Cuculescu's projections from
noncommutative martingale theory in place of classical
stopping times.
There are four terms in Parcet's decomposition:
$$
f=g+b+g_{\mathrm{off}}+b_{\mathrm{off}}.
$$
The additional terms $g_{\mathrm{off}}$ and
$b_{\mathrm{off}}$ are called off-diagonal good and
bad functions, respectively.
They do not enjoy the same properties as the diagonal
parts $g$ and $b$, which creates a substantial obstacle
to applying classical ideas or techniques, and
substantial new ideas are required.
For instance, under the Lipschitz regularity condition,
to deal with $g_{\mathrm{off}}$, Parcet
in \cite{parcet09pseudo} explored the $TT^*$ method
and a new pseudo-localization technique; see also
\cite{meiparcet09pseudo}.
This method does not directly apply under the
H\"ormander condition, so he left the sufficiency
of this condition as an open problem in
\cite[Remark~5.5]{parcet09pseudo}, see also \cite[Remark~1.3]{Parcet2026ICM}.

Subsequently, partial progress was made by Cadilhac,
Conde-Alonso and Parcet in
\cite{cadilhaccondeparcet22spectral},
where they constructed a new Calder\'on--Zygmund
decomposition by modifying $b_{\mathrm{off}}$ so that
$g_{\mathrm{off}}$ disappears.
This enables them to obtain endpoint results under
an $L_2$-H\"ormander regularity condition, which is
stronger than the usual H\"ormander condition but
weaker than the Lipschitz regularity condition.
Under an $L_2$ integral regularity condition and
a maximal strong type assumption, Hong, Lai and Xu
in \cite{honglaixu23sigular}
proved the noncommutative maximal weak type $(1,1)$
estimate for truncations of singular integral operators.
This gives a noncommutative maximal weak type $(1,1)$
inequality for families of linear operators that are
not necessarily positive.

With many applications in mind, the original purpose
of the present paper is to build a noncommutative
version of Z\'o's maximal inequalities for integrable
approximation kernels.
Then we found that the entirely new strategy designed
to overcome the limitations of the classical
vector-valued Calder\'on--Zygmund approach also works
for singular convolution kernels satisfying the
H\"ormander condition.
This provides a positive answer to the previously
mentioned problem for convolution kernels.

Our first result is a noncommutative analogue of Zo's maximal inequalities. The unexplalined notation will be specified in Section~\ref{subsec:distribution-convolutions}.

\begin{thm}\label{thm:joint}
Let $(T_n)_{n\geq1}$ be a family of scalar convolution operators on $\cN$ with kernels $K_n\in L_{1,\mathrm{loc}}(\bR^d\setminus\{0\})$. Suppose that their off-support representations hold on $\mathscr D$ and that
\begin{align}\label{joint hormander}\tag{H}\sup_{y\ne0}\int_{|x|>2|y|}\sup_{n\geq1}|K_n(x-y)-K_n(x)|\,dx\leq C_H.
\end{align}
Assume also that, for some $1<r\leq\infty$,
$$\|(T_nf)_{n\geq1}\|_{L_r(\cN;\ell_\infty)}\leq M_r\|f\|_r.$$
Then $(T_n)_{n\geq1}$ is of weak type $(1,1)$. More precisely, for every $f\in L_1(\cN)$ and $t>0$, there exists a projection $e\in\cN$ such that
$$\sup_{n\geq1}\|eT_n(f)e\|_\infty\leq t,\qquad
t\varphi(e^\perp)\leq C_{d,r}(M_r+C_H)\|f\|_1.$$
\end{thm}


Our second result concerns all sharp truncations of one singular convolution kernel. 

\begin{thm}\label{thm:truncations}
Let $T$ be an $L_2(\cN)$-bounded scalar convolution operator with kernel $K:\bR^d\setminus\{0\}\to\bC$, whose off-support integral formula holds for bounded, spatially compactly supported inputs in $L_1(\cN)$. Suppose that
$$|K(x)|\leq A|x|^{-d},\qquad
\sup_{y\ne0}\int_{|x|>2|y|}|K(x-y)-K(x)|\,dx\leq C_H.$$
For $\varepsilon>0$, set
$$T_\varepsilon f(x)=\int_{|x-y|>\varepsilon}K(x-y)f(y)\,dy.$$
{ Assume that, for some $1<p_0<\infty$,
\begin{align}\label{p0}\|(T_\varepsilon f)_{\varepsilon>0}\|_{L_{p_0}(\cN;\ell_\infty)}
\leq M_{p_0}\|f\|_{p_0}.
\end{align}}
Then $(T_\varepsilon)_{\varepsilon>0}$ is of weak type $(1,1)$. More precisely, for every $f\in L_1(\cN)$ and $t>0$, there exists a projection $e\in\cN$ such that
$$\sup_{\varepsilon>0}\|eT_\varepsilon(f)e\|_\infty\leq t,\qquad
t\varphi(e^\perp)\leq C_{d,p_0}(A+C_H+M_{p_0})\|f\|_1.$$
\end{thm}

In many cases, the assumption \eqref{p0} is redudant, because it can be deduced from the Cotlar inequality; see \cite[Theorem 1.3]
{honglaixu23sigular}. Given one singular integral operator $T$ satisfying all the conditions but not necessarily \eqref{p0} in Theorem \ref{thm:truncations}, applying Theorem \ref{thm:joint} to $T_n=T$ for all $n$ implies the weak type $(1,1)$ estimate of $T$. This resolves, for convolution kernels, the open problem raised in \cite[Remark~5.5]{parcet09pseudo} and reiterated in Parcet's recent ICM survey \cite[Remark~1.3]{Parcet2026ICM}. Finally, the deduction from Theorem \ref{thm:joint} to Theorem \ref{thm:truncations} is not automatic since the kernels of $(T_\varepsilon)_{\varepsilon>0}$ do not satisfy the joint H\"ormander condition \eqref{joint hormander}; see Section \ref{pf of B} for this non trivial passage.

As with Z\'o's result(see, e.g., \cite[Chapter~II, \S\S5.10 and~5.16]{stein93bookharmonic}), Theorem~\ref{thm:joint} gives lots of applications in
Section~\ref{sec:approximation}, such as approximations of the identity, kernels with monotone $L\log L$ majorants, and Poisson integrals on symmetric spaces.
Together with transference techniques, further applications to Fourier multipliers on group von Neumann algebras, quantum Euclidean spaces and quantum tori,
Schur multipliers, and double operator integrals will be presented in Section~\ref{sec:applications}.

\bigskip

\noindent{\bf Strategy and novelities.}
As mentioned previously, the proof of Theorem~\ref{thm:joint} requires substantial new ideas. 
We will exploit the effective decomposition introduced in \cite{cadilhaccondeparcet22spectral}. The main obstruction to using the usual
$L_1$-H\"ormander condition lies in the off-diagonal
terms of the noncommutative decomposition.
Let $b_{\mathrm{off},Q}$ denote the off-diagonal bad
term associated with $Q$ in the decomposition
recalled in Section~\ref{subsec:cuculescu}.
Estimating these terms separately using the
$L_1$-H\"ormander condition and then summing the
resulting bounds would require
$$
\sum_Q\|b_{\mathrm{off},Q}\|_1
\lesssim\|f\|_1,
$$
which is unavailable for the noncommutative
decomposition; see
\cite[Remark~2.4]{cadilhaccondeparcet22spectral}.
Their proof instead uses an annular $L_2$ kernel
condition to estimate the off-diagonal terms;
see also \cite{honglaixu23sigular} for maximal
truncations.

Our new approach will not require the preceding $L_1$ bound by exploring the kernel decomposition.
It consists of two main steps.
First, we establish a weak type $(1,1)$ estimate
for operators represented by dyadic averages.
Second, we regularize the convolution operators
and construct a simultaneous dyadic expansion
of their kernels to verify the assumptions
of this estimate.

\smallskip
\textit{The weak type estimate for dyadic series.}
In Section~\ref{sec:dyadic-weak}, by assuming a maximal strong type $(r,r)$ bound
for some $1<r\leq\infty$, 
we prove a weak type
$(1,1)$ estimate for operator families of the form
$$
\mathscr S_nf
=\sum_{Q\in\cQ}b_{n,Q}\bE_Qf,
\qquad |b_{n,Q}|\leq m_Q,
$$
provided additionally
\begin{align}\label{coefficient bound}
\Gamma:=
\sup_{P\in\cQ}
\mathop{\mathrm{ess\,sup}}_{y\in P}
\sum_{\substack{Q\subseteq P\\y\in Q}}
\frac1{|Q|}
\int_{(\beta P)^c}m_Q(x)\,dx
<\infty,
\end{align}
where $\beta\geq3$ is fixed.
The proof relies on the decomposition in Proposition~\ref{lemma:good-bad-decomposition}.
We briefly explain this decomposition and how it yields the weak type $(1,1)$ estimate.
First, we use Cuculescu's construction at height $t>0$ to obtain the diagonal good function $g$. For any $\lambda>0$, independently of $t$, we then
construct a projection $w$ such that
$$
w\mathscr S_n(f-g)w=B_n+B_{\ell,n}+B_{r,n},
$$
where the three bad terms depend on $w$. Here, we choose $w$ so that each bad term has the operator norm
at most $\lambda/4$ for every $n$, while keeping $\varphi(1-w)$ small. To deduce the weak type estimate, we apply the assumed
strong type bound to the good part, and choose $t$ appropriately in terms of $\lambda$.

We now turn to the construction of $w$, where our treatment of  bad terms differs substantially from
previous approaches; see, for example, \cite{cadilhaccondeparcet22spectral}. Although the good function $g$ comes from the
standard noncommutative Calder\'on--Zygmund decomposition recalled in Lemma~\ref{lemma:cuculescu-cz}, we use the coefficient bound to choose one projection $w$ controlling the three bad terms for every $n$, with small $\varphi(1-w)$. The main difficulty lies in the off-diagonal bad terms $B_{\ell,n}$ and $B_{r,n}$, which we estimate without summing the $L_1$ norms of the individual off-diagonal bad terms.

To explain this construction, write $f_Q=\bE_Qf$, let $q_Q$ be the Cuculescu projection at height $t$, and put
$$
\Pi_Q(x)
=\sum_{P\in\mathcal P(Q)}
\dsone_{(\beta P)^c}(x)p_P,
$$
where $\mathcal P(Q)$ and $p_P$ are the stopping
ancestors and difference projections defined in
the proof of Proposition~\ref{lemma:good-bad-decomposition}.
The coefficient bound \eqref{coefficient bound} gives
$$
\varphi\left(\sum_Qm_Q\Pi_Q\right)
\leq\frac{\Gamma}{t}\|f\|_1,
\qquad
\sum_Q\varphi(m_Q\Pi_Qf_Q\Pi_Q)
\leq\Gamma\|f\|_1.
$$
After a spectral cut, we obtain a projection $w_0$
with $\varphi(1-w_0)\leq C_{d,\beta}\|f\|_1/t$,
on which the left off-diagonal sum factors as follows:
$$
\sum_Qb_{n,Q}w_0\Pi_Qf_Qq_Qw_0=R_nC,
$$
where
$$
R_n=\left(
\frac{b_{n,Q}}{\sqrt{m_Q}}w_0\Pi_Q
\right)_Q,
\qquad
C=\left(
\sqrt{m_Q}\Pi_Qf_Qq_Qw_0
\right)_Q^{\mathrm{column}},
$$
with the quotient defined to be zero when $m_Q=0$.
The preceding estimates and $f_Qq_Qf_Q\leq tf_Q$ give
$$
(R_nC)^*(R_nC)\leq\Gamma C^*C,
\qquad
\varphi(C^*C)\leq t\Gamma\|f\|_1.
$$
The same bound holds for the adjoint of the right
off-diagonal sum. Therefore, starting from $w_0$, we take a spectral cut of $C^*C$
to control both off-diagonal sums at height $\lambda/4$,
with additional trace cost at most
$16t\Gamma^2\lambda^{-2}\|f\|_1$.
We then take a further cut to control $B_n$,
obtaining the desired projection $w\leq w_0$.

\smallskip
\textit{The simultaneous kernel decomposition and dyadic representations of operators.}
To apply this abstract estimate to the operators
in Theorem~\ref{thm:joint}, we first regularize them such that the regularized kernels are bounded and Lipschitz; see Lemma \ref{lem:regularized-kernels}.

For a fixed regularization parameter $\varepsilon>0$, write
$T_{n,\varepsilon}$ and $K_{n,\varepsilon}$ for the regularized operators
and their kernels.
The main construction in
Section~\ref{sec:kernel-decomposition} gives
one family $\cD$ of dyadic cubes and coefficients
$a_{n,Q}$ such that
$
K_{n,\varepsilon}=\sum_{Q\in\cD}a_{n,Q}\dsone_Q,
$
and, for every $P\in\cD$,
$$
\sum_{\substack{Q\in\cD\\Q\subsetneq P}}
|Q|\sup_n|a_{n,Q}|
\lesssim
\int_P\sup_n|K_{n,\varepsilon}-\bE_PK_{n,\varepsilon}|;
$$
see Proposition~\ref{prop:dyadic-majorants}.

For each cube $Q\in\cD$ in the kernel expansion
and each $x\in\bR^d$, we partition $x-Q$ into maximal standard dyadic cubes.
This expresses $\dsone_Q*f(x)$ in terms of
the dyadic averages of $f$.
This converts the kernel expansion into
$$T_{n,\varepsilon}f
=\sum_{Q\in\cQ}b_{n,Q}\bE_Qf,
\qquad |b_{n,Q}|\leq m_Q.
$$
The coefficient estimates for the kernel expansion,
together with the joint H\"ormander condition
and the geometry of translated cubes, yield
$
\Gamma\leq C_{d,r}(M_r+C_H);
$
see Proposition~\ref{prop:kernel-pairing}.
This is the bound needed in
Theorem~\ref{prop:exterior-endpoint}.
Its uniformity in the regularization parameter
allows us to recover the original operators
and complete the proof of Theorem~\ref{thm:joint} in Section~\ref{sec:maximal}.

\medskip
\noindent\textbf{Limitation of the method.}
Convolution allows us to combine the kernel decomposition
with the geometry of translated cubes to obtain
$\Gamma<\infty$.
For general kernels, the coefficients depend on $x$,
and we do not know how to deduce the same bound
from the H\"ormander condition.
\bigskip

\noindent\textbf{Organization of the paper.}
Section~\ref{sec:preliminaries} collects the definitions
and preliminary results used below.
Section~\ref{sec:dyadic-weak} proves a weak type $(1,1)$
estimate for dyadic series and the decomposition
used in its proof.
Section~\ref{sec:kernel-decomposition} constructs the
simultaneous kernel decomposition and expresses the
convolution operators in terms of dyadic averages.
Section~\ref{sec:maximal} discusses the joint
H\"ormander condition and regularization, and proves
Theorems~\ref{thm:joint} and~\ref{thm:truncations}.
Section~\ref{sec:approximation} studies maximal inequalities
and approximation of the identity.
It first treats integrable kernels and dyadic approximation,
then monotone $L\log L$ majorants, and finally applications
to nilpotent groups and Poisson kernels.
Section~\ref{sec:applications} uses transference to obtain
maximal inequalities for Fourier multipliers, Schur
multipliers and double operator integrals, together with
bilateral almost uniform convergence of dyadic Fourier
summation methods on $L_1$ of group von Neumann algebras.

\section{Preliminaries}\label{sec:preliminaries}

In this section, we introduce the notation and recall basic facts about noncommutative measure spaces and convolution operators. We also recall the noncommutative Calder\'on--Zygmund decomposition and transference results used below.
\subsection{Noncommutative measure spaces}

\subsubsection{Noncommutative $L_p$-spaces}
Let $\cM$ be a semifinite von Neumann algebra equipped
with a normal faithful semifinite trace $\tau$.
Denote by $\cM_+$ its positive cone and by $S(\cM)_+$
the set of all $x\in\cM_+$ whose support projection
has finite trace. Let $S(\cM)$ be the linear span
of $S(\cM)_+$.
For $1\leq p<\infty$, the noncommutative space
$L_p(\cM)$ is the completion of $S(\cM)$ under
the norm
$$
\|x\|_p=\tau(|x|^p)^{1/p},
\qquad |x|=(x^*x)^{1/2}.
$$
As usual, $L_\infty(\cM)=\cM$ with the operator norm.

A closed densely defined operator $x$ affiliated
with $\cM$ is called $\tau$-measurable if
$$
\tau\bigl(\dsone_{(t,\infty)}(|x|)\bigr)<\infty
\qquad\text{for some }t>0.
$$
We denote the algebra of these operators by $L_0(\cM)$.
Under the usual identification,
$$
L_p(\cM)
=\{x\in L_0(\cM):\tau(|x|^p)<\infty\},
\qquad 1\leq p<\infty.
$$
And the weak $L_p$-space
$L_{p,\infty}(\cM)$ consists of those $x\in L_0(\cM)$
for which
$$
\|x\|_{p,\infty}
:=\sup_{t>0}t\,d_x(t)^{1/p}<\infty,
$$
where $d_x(t)=\tau(\dsone_{(t,\infty)}(|x|))$
is the distribution function of $x$.
This expression defines a quasi-norm, and
Chebyshev's inequality gives
$$
\|x\|_{p,\infty}\leq\|x\|_p,
\qquad x\in L_p(\cM).
$$
We refer to~\cite{pisierxu03nclp} for further details
on noncommutative $L_p$-spaces.

\subsubsection{Vector-valued noncommutative $L_p$-spaces}
For $1\leq p\leq\infty$, let
$L_p(\cM;\ell_\infty(\Lambda))$ be the space of families
$x=(x_\lambda)_{\lambda\in\Lambda}$ that admit a factorization
$$
x_\lambda=ay_\lambda b,
\qquad \lambda\in\Lambda,
$$
where $a,b\in L_{2p}(\cM)$ are independent of $\lambda$
and $(y_\lambda)_{\lambda\in\Lambda}$ is a bounded family
in $\cM$. Here $\Lambda$ is an arbitrary index set. And
the norm of $x$ in $L_p(\cM;\ell_\infty(\Lambda))$
is defined by
$$
\|x\|_{L_p(\cM;\ell_\infty(\Lambda))}
=\inf\left\{
    \|a\|_{2p}
    \sup_{\lambda\in\Lambda}\|y_\lambda\|_\infty
    \|b\|_{2p}
\right\},
$$
where the infimum runs over all such factorizations.
For $p=\infty$, this reduces to
$$
\|x\|_{L_\infty(\cM;\ell_\infty(\Lambda))}
=\sup_{\lambda\in\Lambda}\|x_\lambda\|_\infty.
$$
We refer to
\cite{pisier98noncommutativeLp,junge02doob,jungexu07ergodic}
for further details.

\begin{remark}\label{rk:MaxFunct}
Let $x=(x_\lambda)_{\lambda\in\Lambda}$ be a family
of selfadjoint operators in $L_p(\cM)$.
Then $x\in L_p(\cM;\ell_\infty(\Lambda))$ if and only if
there exists $a\in L_p(\cM)_+$ such that
$-a\leq x_\lambda\leq a$ for every $\lambda\in\Lambda$.
Moreover,
$$
\|x\|_{L_p(\cM;\ell_\infty(\Lambda))}
=\inf\left\{
    \|a\|_p:
    a\in L_p(\cM)_+,\ 
    -a\leq x_\lambda\leq a
    \ \text{for all }\lambda\in\Lambda
\right\}.
$$
In particular, if each $x_\lambda$ is positive, then
$$
\|x\|_{L_p(\cM;\ell_\infty(\Lambda))}
=\inf\left\{
    \|a\|_p:
    a\in L_p(\cM)_+,\ 
    x_\lambda\leq a
    \ \text{for all }\lambda\in\Lambda
\right\}.
$$
These properties can be found in
\cite[Remark~4.1]{chenxuyin13harmonic}.
\end{remark}


For $1\leq p<\infty$, let
$\Lambda_{p,\infty}(\cM;\ell_\infty(\Lambda))$
denote the space of families
$x=(x_\lambda)_{\lambda\in\Lambda}\subset L_0(\cM)$
for which
$$
\|x\|_{\Lambda_{p,\infty}(\cM;\ell_\infty(\Lambda))}
:=\sup_{t>0}t\inf\left\{
    \tau(1-e)^{1/p}:
    e\in\cP(\cM),\
    \sup_{\lambda\in\Lambda}\|ex_\lambda e\|_\infty\leq t
\right\}<\infty,
$$
where $\cP(\cM)$ denotes the set of projections in $\cM$.
This expression defines the weak maximal quasi-norm
of the family $x$.
If each $x_\lambda$ is selfadjoint, the condition
$\sup_{\lambda\in\Lambda}\|ex_\lambda e\|_\infty\leq t$
is equivalent to
$$
-te\leq ex_\lambda e\leq te
\qquad\text{for every }\lambda\in\Lambda.
$$
We refer to~\cite[Section~2.2]{honglaixu23sigular}
for further details.

\begin{definition}\label{def:maximal-types}
Let $1\leq p\leq\infty$, and let
$S=(S_\lambda)_{\lambda\in\Lambda}$ be a family of maps
from $L_p(\cM)$ into $L_0(\cM)$.
For $x\in L_p(\cM)$, write
$Sx=(S_\lambda x)_{\lambda\in\Lambda}$.
We say that $S$ is of strong type $(p,p)$ if there
exists a constant $C_p$ such that
$$
\|Sx\|_{L_p(\cM;\ell_\infty(\Lambda))}
\leq C_p\|x\|_p,
\qquad x\in L_p(\cM).
$$
When $p<\infty$, we say that $S$ is of weak type
$(p,p)$ if there exists a constant $C_p$ such that
$$
\|Sx\|_{\Lambda_{p,\infty}(\cM;\ell_\infty(\Lambda))}
\leq C_p\|x\|_p,
\qquad x\in L_p(\cM).
$$
\end{definition}

The following consequence of the factorization defining
$L_p(\cM;\ell_\infty)$ and Chebyshev's inequality will be used below.

\begin{lemma}\label{lemma:factor-cut}
Let $1<p<\infty$ and $x=(x_n)_{n\geq1}\in L_p(\cM;\ell_\infty)$.
For every $t>0$, there exists $e\in\cP(\cM)$ such that
$$
\sup_n\|ex_ne\|_\infty\leq t,
\qquad
\tau(1-e)\leq
\frac{2^{p+1}}{t^p}\|x\|_{L_p(\cM;\ell_\infty)}^p.
$$
\end{lemma}

We recall the noncommutative Marcinkiewicz interpolation theorem
of Junge and Xu.

\begin{lemma}[{\cite[Theorem~3.1]{jungexu07ergodic}}]
\label{lemma:positive-interpolation}
Let $1\leq p_0<p_1\leq\infty$, and let $S=(S_n)_{n\geq1}$
be a family of maps from $L_{p_0}(\cM)_++L_{p_1}(\cM)_+$
into $L_0(\cM)_+$ satisfying
$$
S_n(x+y)\leq S_n(x)+S_n(y)
\qquad(n\geq1).
$$
Suppose that $S$ is of weak type $(p_0,p_0)$ and strong type
$(p_1,p_1)$ on the positive cones, with constants $B_0$ and
$B_1$, respectively. Then, for every $p_0<p<p_1$,
$$
\|Sx\|_{L_p(\cM;\ell_\infty)}
\leq C B_0^{1-\theta}B_1^\theta
\left(\frac1{p_0}-\frac1p\right)^{-2}\|x\|_p,
\qquad x\in L_p(\cM)_+,
$$
where
$$
\frac1p=\frac{1-\theta}{p_0}+\frac{\theta}{p_1},
$$
and $C$ is a universal constant.
\end{lemma}

\subsubsection{Bilateral almost uniform convergence}

A sequence $(x_n)\subset L_0(\cM)$ converges bilaterally
almost uniformly to $x\in L_0(\cM)$ if, for every
$\varepsilon>0$, there is a projection $e\in\cP(\cM)$
such that
$$
\tau(1-e)<\varepsilon,
\qquad
\|e(x_n-x)e\|_\infty\longrightarrow0.
$$
We write $x_n\to x$ b.a.u.; see
\cite[Definition~6.1(i)]{jungexu07ergodic}.

For operator-valued functions on $\bR^d$, put
$$
\cN=L_\infty(\bR^d)\overline\otimes\cM,
\qquad
\varphi=\int_{\bR^d}dx\otimes\tau.
$$
Suppose that the restrictions of $f_n$ and $f$ to
every bounded measurable set are trace-measurable.
We say that $f_n\to f$ locally b.a.u. if, for every
bounded measurable set $E\subset\bR^d$ and every
$\varepsilon>0$, there is a projection
$e\leq p_E:=\dsone_E\otimes1$ such that
$$
\varphi(p_E-e)<\varepsilon,
\qquad
\|e(f_n-f)e\|_\infty\longrightarrow0.
$$
In particular, this definition applies to functions in
$L_{1,\mathrm{loc}}(\bR^d;L_1(\cM))$.

\subsection{Convolution operators}
\label{subsec:distribution-convolutions}

For $1\leq p<\infty$, we identify
$L_p(\cN)$ with $L_p(\bR^d;L_p(\cM))$.
Let $K\in L_1(\bR^d)$ and
$f\in L_p(\cN)$.
The convolution of $K$ and $f$ is defined by
$$
K*f=\int_{\bR^d}K(y)f(\,\cdot-y)\,dy,
$$
where the integral is a Bochner integral in $L_p(\cN)$.
For $p=\infty$, convolution by $K$ is defined by
weak$^*$ duality.
These definitions agree on intersections and satisfy
\begin{equation}\label{eq:integrable-kernel-bound}
    \|K*f\|_p\leq\|K\|_1\|f\|_p,
    \qquad 1\leq p\leq\infty.
\end{equation}

More generally, let $\sigma$ be a scalar distribution
on $\bR^d$ whose restriction to $\bR^d\setminus\{0\}$
is represented by
$K\in L_{1,\mathrm{loc}}(\bR^d\setminus\{0\})$.
We define the associated convolution operator $T$
on the dense subspace
$$
\mathscr D=C_c^\infty(\bR^d)\otimes S(\cM)
$$
of $L_1(\cN)$, where the tensor product is algebraic.
More specifically, for
$f=\sum_{k=1}^m u_k\otimes a_k\in\mathscr D$, we set
$$
Tf=\sum_{k=1}^m(\sigma*u_k)\otimes a_k,
\qquad
(\sigma*u_k)(x)
=\langle\sigma,u_k(x-\cdot)\rangle.
$$
Consequently,
\begin{equation}\label{eq:off-support-formula}
    Tf(x)=\int_{\bR^d}K(x-y)f(y)\,dy
    \quad\text{for almost every }
    x\notin\supp_{\bR^d}f.
\end{equation}

\subsection{Noncommutative Calder\'on--Zygmund decomposition}
\label{subsec:cuculescu}

The first noncommutative Calder\'on--Zygmund decomposition
was introduced by Parcet~\cite{parcet09pseudo}. 
But the decomposition to be used in the present paper is closer to that in \cite[Sections~1 and~2.1]{cadilhaccondeparcet22spectral}. 

For $j\in\bZ$, let $\cQ_j$ denote the standard dyadic
cubes of side length $2^{-j}$:
$$
\cQ_j=\{2^{-j}(m+[0,1)^d):m\in\bZ^d\},
\qquad
\cQ=\bigcup_{j\in\bZ}\cQ_j.
$$
For an axis-parallel cube $Q$, denote its side length
and centre by $\ell(Q)$ and $c_Q$, respectively.
For $c>0$, write $cQ$ for the concentric cube of side
length $c\ell(Q)$, and denote the average of $f$ over $Q$ by
$$
\bE_Qf=\frac1{|Q|}\int_Qf(y)\,dy.
$$
Let $\cN_j$ be the von Neumann subalgebra of $\cN$
consisting of functions constant almost everywhere
on each $Q\in\cQ_j$.
The associated trace-preserving conditional expectation is
\begin{equation}\label{eq:dyadic-expectation}
    \bE_jf
    =\sum_{Q\in\cQ_j}\dsone_Q\otimes\bE_Qf.
\end{equation}
These expectations are contractive on $L_1(\cN)$
and satisfy
\begin{equation}\label{eq:dyadic-regularity}
    \bE_jf\leq2^d\bE_{j-1}f,
    \qquad f\in L_1(\cN)_+.
\end{equation}

Fix $f\in L_1(\cN)_+$, $t>0$ and $j_0\in\bZ$,
and write $f_j=\bE_jf$.
Following Cuculescu's construction
\cite{cuculescu71martingales}, as presented in
\cite[Section~1]{cadilhaccondeparcet22spectral},
define
\begin{equation}\label{eq:cuculescu-projections}
    q_{j_0-1}=1,\qquad
    q_j=\dsone_{[0,t]}(q_{j-1}f_jq_{j-1}),
    \qquad
    p_j=q_{j-1}-q_j
    \quad(j\geq j_0),
\end{equation}
where the spectral projection is taken in
$q_{j-1}\cN_jq_{j-1}$.
Then $q_j\in\cN_j$, $q_j\leq q_{j-1}$, and
\begin{equation}\label{eq:cuculescu-compressions}
    q_jf_jq_j\leq tq_j,\qquad
    p_jf_jq_j=q_jf_jp_j=0,\qquad
    p_jf_jp_j\geq tp_j.
\end{equation}
The projections $p_j$ are pairwise orthogonal, with
$q_{j_1}+\sum_{j=j_0}^{j_1}p_j=1$
for every $j_1\geq j_0$, and satisfy
\begin{equation}\label{eq:cuculescu-trace}
    t\sum_{j\geq j_0}\varphi(p_j)\leq\|f\|_1.
\end{equation}

For each $j_1\geq j_0$, we have the decomposition
\begin{equation}\label{eq:finite-cz-decomposition}
    f_{j_1}=g+b_{\mathrm d}+b_{\mathrm{off}},
\end{equation}
where
\[
g=q_{j_1}f_{j_1}q_{j_1}
  +\sum_{j=j_0}^{j_1}p_jf_jp_j,
\qquad
b_{\mathrm d}
  =\sum_{j=j_0}^{j_1}p_j(f_{j_1}-f_j)p_j,
\]
and
\[
b_{\mathrm{off}}
  =\sum_{j=j_0}^{j_1}
   \bigl[p_j(f_{j_1}-f_j)q_j
         +q_j(f_{j_1}-f_j)p_j\bigr].
\]
The following properties can be found in
\cite[Equation~(1.1) and Lemma~1.1]
{cadilhaccondeparcet22spectral}.

\begin{lemma}\label{lemma:cuculescu-cz}
The terms in \eqref{eq:finite-cz-decomposition} satisfy
\begin{equation}\label{eq:finite-cz-L1}
    g\geq0,\qquad
    \|g\|_1=\|f\|_1,\qquad
    \sum_{j=j_0}^{j_1}
    \|p_j(f_{j_1}-f_j)p_j\|_1
    \leq2\|f\|_1.
\end{equation}
Moreover, for $j_0\leq j\leq j_1$,
\[
    \begin{aligned}
        \bE_j\bigl(p_j(f_{j_1}-f_j)p_j\bigr)&=0,\\
        \bE_j\bigl(
            p_j(f_{j_1}-f_j)q_j
            +q_j(f_{j_1}-f_j)p_j
        \bigr)&=0.
    \end{aligned}
\]
If, in addition, $f_{j_0-1}\leq t1$, then
\begin{equation}\label{eq:finite-cz-Linfty}
    p_jf_jp_j\leq2^dtp_j,
    \qquad
    \|g\|_\infty\leq2^dt.
\end{equation}
\end{lemma}

Since $p_j\in\cN_j$, there are projections
$p_Q\in\cP(\cM)$ such that
$$
p_j=\sum_{Q\in\cQ_j}\dsone_Q\otimes p_Q.
$$
For $\beta\geq1$, define
\begin{equation}\label{eq:enlarged-cuculescu-projection}
    \zeta
    =1-\bigvee_{j\geq j_0}\bigvee_{Q\in\cQ_j}
       (\dsone_{\beta Q}\otimes p_Q).
\end{equation}
The following properties are given in
\cite[Lemma~3.4]{cadilhac18weak}
for odd positive integer dilation factors.
The same proof applies to every $\beta\geq1$.

\begin{lemma}\label{lemma:enlarged-cuculescu}
The projection $\zeta$ satisfies
\[
    \varphi(1-\zeta)
    \leq\frac{\beta^d}{t}\|f\|_1.
\]
Moreover, for every $Q\in\cQ_j$ with $j\geq j_0$,
\[
    \zeta(x)p_Q=p_Q\zeta(x)=0
    \quad\text{for almost every }x\in\beta Q.
\]
\end{lemma}

\subsection{Auxiliary lemmas and transference}

We collect three lemmas used in the maximal estimates below,
and then recall a transference principle for positive
convolution operators.

\subsubsection{Positive operator matrices and compression estimates}

We first prove a factorization for positive operator matrices.
The following factorization is classical for bounded operators;
see \cite[Theorem~5.11]{MKX2019}. We include a proof that also
applies to measurable operators.

\begin{lemma}\label{lemma:positive-block}
Let $1\leq p\leq\infty$. Suppose that
$u,v\in L_p(\cM)_+$ and $x\in L_p(\cM)$ satisfy
$$
\begin{pmatrix}
    u & x\\
    x^* & v
\end{pmatrix}\geq0.
$$
Then there is a contraction $z\in\cM$ such that
$x=u^{1/2}zv^{1/2}$.
In particular, if $u,v\leq te$ for some $t\geq0$ and
$e\in\cP(\cM)$, then $\|x\|_\infty\leq t$.
\end{lemma}

\begin{proof}
Write
$$
R=
\begin{pmatrix}
    u & x\\
    x^* & v
\end{pmatrix}^{1/2}
=(R_1\ R_2),
$$
where $R_1$ and $R_2$ are the columns of $R$. Then
$$
R_1^*R_1=u,\qquad
R_2^*R_2=v,\qquad
R_1^*R_2=x.
$$
By polar decomposition,
$$
R_1=w_1u^{1/2},\qquad R_2=w_2v^{1/2},
$$
where $w_1,w_2\in M_{2,1}(\cM)$ are partial isometries.
Here $M_{2,1}(\cM)$ denotes the space of $2\times1$ matrices
with entries in $\cM$, viewed as operators from $H$ to
$H\oplus H$, where $\cM$ acts on $H$.

Set $z=w_1^*w_2$. Since
$$
z^*z=w_2^*w_1w_1^*w_2\leq w_2^*w_2\leq1,
$$
we have $\|z\|_\infty\leq1$, and
$$
x=R_1^*R_2=u^{1/2}zv^{1/2}.
$$
If $u,v\leq te$, then
$$
\|x\|_\infty
\leq\|u^{1/2}\|_\infty
     \|z\|_\infty
     \|v^{1/2}\|_\infty
\leq t.
$$
\end{proof}

\begin{lemma}\label{lemma:bounded-closure}
Let $t\geq0$. Suppose that $x_j\in L_1(\cM)\cap\cM$,
$x_j\to x$ in $L_1(\cM)$, and $\|x_j\|_\infty\leq t$.
Then $x\in\cM$ and $\|x\|_\infty\leq t$.
If also $x_j\geq0$ for all $j$, then $x\geq0$.
\end{lemma}

\begin{proof}
Let $x=w|x|$ be the polar decomposition of $x$.
For $\varepsilon>0$, put
$q=\dsone_{(t+\varepsilon,\infty)}(|x|)$.
Since $x\in L_1(\cM)$,
$$
\tau(q)\leq\frac{\|x\|_1}{t+\varepsilon}<\infty.
$$
Moreover, $\|qw^*\|_\infty\leq1$ and
$\|qw^*\|_1\leq\tau(q)$. Hence
\begin{align*}
    (t+\varepsilon)\tau(q)
    \leq\tau(q|x|)
    \leq|\tau(qw^*x_j)|+\|x-x_j\|_1
    \leq t\tau(q)+\|x-x_j\|_1.
\end{align*}
Letting $j\to\infty$ gives $\tau(q)=0$, and hence $q=0$.
Since $\varepsilon>0$ is arbitrary, $|x|\leq t1$.
The last assertion follows from the norm closedness of
$L_1(\cM)_+$.
\end{proof}

The next lemma follows by adapting the compactness argument
in the proof of \cite[Lemma~3.2]{Hong2017}.

\begin{lemma}\label{lemma:positive-closure}
Let $(x_n)_{n\geq1}\subset L_1(\cM)_+$.
Suppose that there are projections $e_N\in\cM$ and numbers
$t_N\to t>0$ and $D_N\to D<\infty$ such that
$$
e_Nx_ne_N\leq t_Ne_N
\quad(1\leq n\leq N),
\qquad
\tau(1-e_N)\leq D_N.
$$
Then there is a projection $e\in\cM$ satisfying
$$
ex_ne\leq2te
\quad(n\geq1),
\qquad
\tau(1-e)\leq2D.
$$
\end{lemma}

\begin{proof}
By the Banach--Alaoglu theorem, a subnet
$(e_{N_\alpha})$ converges weak$^*$ to some
$a\in\cM$ with $0\leq a\leq1$.
By the weak$^*$ lower semicontinuity of $\tau$, we have
$$
\tau(1-a)
\leq\liminf_\alpha\tau(1-e_{N_\alpha})
\leq D.
$$
Set $e=\dsone_{[1/2,1]}(a)$, we obtain $\tau(1-e)\leq\tau(1-2a)\leq2D$.

Fix $n\geq1$ and put
$x_{n,k}=f_k(x_n)$, where $f_k(s)=\min\{s,k\}$ for $s\geq0$. we have $0\leq x_{n,k}\leq x_n$, and
$$
x_{n,k}\in L_1(\cM)\cap\cM,\qquad
\|x_{n,k}-x_n\|_1\longrightarrow0.
$$
Moreover, for fixed $k$ and all sufficiently large $\alpha$, we get
$$
e_{N_\alpha}x_{n,k}e_{N_\alpha}
\leq t_{N_\alpha}e_{N_\alpha}.
$$
Note that for $b\in\cM_+$, $p\in\cP(\cM)$ and $c\geq0$, we have
\begin{equation}\label{eq:positive-compression-equivalence}
    pbp\leq cp
    \quad\Longleftrightarrow\quad
    b^{1/2}pb^{1/2}\leq c1,
\end{equation}
since both inequalities are equivalent to
$\|b^{1/2}p\|_\infty^2\leq c$.
Thus, we get
$$
x_{n,k}^{1/2}e_{N_\alpha}x_{n,k}^{1/2}
\leq t_{N_\alpha}1.
$$
Since $x_{n,k}^{1/2}$ is bounded, 
using $e\leq2a$ and passing to the weak$^*$ limit, we obtain
$$
x_{n,k}^{1/2}ex_{n,k}^{1/2}
\leq2x_{n,k}^{1/2}ax_{n,k}^{1/2}
\leq2t1.
$$
Another application of
\eqref{eq:positive-compression-equivalence} gives
$ex_{n,k}e\leq2te$.

Finally, note that
$$
\|ex_{n,k}e-ex_ne\|_1
\leq\|x_{n,k}-x_n\|_1\longrightarrow0,
$$
we have $\|ex_ne\|_\infty\leq2t$ by Lemma~\ref{lemma:bounded-closure}.
Since $ex_ne$ is positive and supported on $e$, we have $ex_ne\leq2te$.
The projection $e$ is independent of $n$.
\end{proof}

\subsubsection{Transference for positive convolution operators}

We now recall a transference principle that transfers maximal
inequalities for positive convolution operators on $\bR^d$
to averages associated with trace-preserving actions.
The following form follows from the proofs of
\cite[Theorems~3.1 and~3.3]{hongwangliao2017noncommutative}.
Although those theorems are stated for probability measures,
their proofs apply equally to finite positive measures
$K_n(y)\,dy$, without requiring $\int_{\bR^d}K_n=1$.

Let $\alpha=(\alpha_y)_{y\in\bR^d}$ be a family of
$*$-automorphisms of $\cM$ satisfying
$\alpha_0=\mathrm{id}_{\cM}$ and
$\alpha_{y+z}=\alpha_y\circ\alpha_z$ for all $y,z\in\bR^d$.
Assume that $\tau(\alpha_y(x))=\tau(x)$ for every
$x\in\cM_+$ and $y\in\bR^d$.
For $(K_n)_{n\geq1}\subset L_1(\bR^d)_+$, define
$$
T_nF(s)=\int_{\bR^d}K_n(y)F(s+y)\,dy,
\qquad
T_n^\alpha x=\int_{\bR^d}K_n(y)\alpha_y(x)\,dy.
$$
The following statement concerns these families on the
positive cones of $L_p(\cN)$ and $L_p(\cM)$, respectively.

\begin{lemma}\label{lemma:positive-transference}
For $1<p<\infty$, a strong type $(p,p)$ bound for $(T_n)$
with constant $C_p$ implies the same bound for
$(T_n^\alpha)$.

Similarly, a weak type $(1,1)$ bound for $(T_n)$
with constant $C_1$ implies a weak type $(1,1)$ bound
for $(T_n^\alpha)$ with constant $CC_1$,
where $C$ is universal.
\end{lemma}

\section{A weak type \texorpdfstring{$(1,1)$}{(1,1)}
estimate for dyadic series}
\label{sec:dyadic-weak}

In this section, we prove a weak type $(1,1)$ estimate
for operator families represented by dyadic averages.
We first state and prove the theorem, using a decomposition
whose proof is given in the second subsection.

\subsection{The weak type estimate}
\label{subsec:dyadic-weak-estimate}

Let $\mathscr D_0$ denote the set of functions
$f\in L_1(\cN)\cap L_\infty(\cN)$ with bounded support.
Fix $\beta\geq3$, and suppose that a countable family
$(\mathscr S_n)_{n\geq1}$ admits the representation
\[
    \mathscr S_nf
    =\sum_{Q\in\cQ}b_{n,Q}\bE_Qf,
    \qquad
    |b_{n,Q}|\leq m_Q,
    \qquad f\in\mathscr D_0,
\]
with absolute convergence in local $L_1$.
Here $b_{n,Q}$ are measurable scalar functions, and
$m_Q$ are nonnegative measurable functions independent
of $n$ and finite almost everywhere.
Put
\begin{equation}\label{eq:exterior-budget}
    \Gamma
    =\sup_{P\in\cQ}
      \mathop{\mathrm{ess\,sup}}_{y\in P}
      \sum_{\substack{Q\subseteq P\\y\in Q}}
      \frac1{|Q|}
      \int_{(\beta P)^c}m_Q(x)\,dx.
\end{equation}

\begin{theorem}\label{prop:exterior-endpoint}
Suppose that $\Gamma<\infty$ and that, for some
$1<r\leq\infty$,
$$
\|(\mathscr S_nf)_n\|_{L_r(\cN;\ell_\infty)}
\leq M_r\|f\|_r,
\qquad f\in L_r(\cN).
$$
Then, for every $f\in\mathscr D_0$ and $\lambda>0$,
there is a projection $e\in\cN$ such that
$$
\sup_n\|e\mathscr S_n(f)e\|_\infty\leq\lambda,
\qquad
\lambda\varphi(1-e)
\leq C_{d,\beta,r}(M_r+\Gamma)\|f\|_1.
$$
\end{theorem}

The theorem follows from the following decomposition proposition, whose proof
is given in the next subsection. As in the earlier noncommutative Calder\'on--Zygmund
decomposition recalled in Lemma~\ref{lemma:cuculescu-cz},
we use Cuculescu's projections to construct a good
function $g$. 
Here we rearrange the terms arising from
$b_{\mathrm d}$ and $b_{\mathrm{off}}$ into three sums after applying
$\mathscr S_n$.
The bound \eqref{eq:exterior-budget} allows us to control
these sums on one projection $w$ for every $n$,
without requiring a bound for the sum of the $L_1$ norms
of the individual off-diagonal bad terms.

\begin{proposition}\label{lemma:good-bad-decomposition}
Suppose that $\Gamma<\infty$.
For every $f\in\mathscr D_0$ with $f\geq0$ and every
$t,\lambda>0$, there exist $g\in\mathscr D_0$,
a projection $w\in\cN$, and operators
$B_n,B_{\ell,n},B_{r,n}\in w\cN w$ such that
\begin{equation}\label{eq:operator-good-bad}
    w\mathscr S_n(f)w
    =w\mathscr S_n(g)w+B_n+B_{\ell,n}+B_{r,n},
    \qquad n\geq1.
\end{equation}
Moreover,
\begin{equation}\label{eq:good-input-bounds}
    0\leq g\leq2^dt1,
    \qquad
    \|g\|_1=\|f\|_1,
\end{equation}
\begin{equation}\label{eq:bad-output-bounds}
    \sup_{n\geq1}
    \max\bigl\{
        \|B_n\|_\infty,\,
        \|B_{\ell,n}\|_\infty,\,
        \|B_{r,n}\|_\infty
    \bigr\}
    \leq\lambda/4,
\end{equation}
and
\begin{equation}\label{eq:bad-projection-bound}
    \varphi(1-w)
    \leq C_{d,\beta}
    \left(
        \frac1t+\frac{\Gamma}{\lambda}
        +\frac{t\Gamma^2}{\lambda^2}
    \right)\|f\|_1.
\end{equation}
\end{proposition}

\begin{proof}[Proof of Theorem~\ref{prop:exterior-endpoint}]
We first assume that $f\geq0$.
If $M_r+\Gamma=0$, the conclusion follows immediately
from the assumed strong type estimate.
Otherwise, apply Proposition~\ref{lemma:good-bad-decomposition}
with parameters $t,\lambda>0$.

Suppose first that $1<r<\infty$.
By \eqref{eq:good-input-bounds}, we have
$$
\|g\|_r^r
\leq\|g\|_\infty^{r-1}\|g\|_1
\leq(2^dt)^{r-1}\|f\|_1.
$$
Hence
$$
\|(\mathscr S_ng)_n\|_{L_r(\cN;\ell_\infty)}^r
\leq M_r^r(2^dt)^{r-1}\|f\|_1.
$$
Lemma~\ref{lemma:factor-cut} gives a projection $e_g$
such that
$$
\sup_n\|e_g\mathscr S_n(g)e_g\|_\infty
\leq\lambda/4,
\qquad
\varphi(1-e_g)
\leq C_{d,r}M_r^rt^{r-1}\lambda^{-r}\|f\|_1.
$$
Set $e=w\wedge e_g$.
By \eqref{eq:operator-good-bad} and
\eqref{eq:bad-output-bounds}, we get
$$
\sup_n\|e\mathscr S_n(f)e\|_\infty\leq\lambda.
$$
Moreover, we obtain
\begin{align*}
    \lambda\varphi(1-e)
    &\leq
      \lambda\varphi(1-w)+\lambda\varphi(1-e_g)\\
    &\leq
      C_{d,\beta,r}
      \left[
          \frac{\lambda}{t}
          +\Gamma
          +\frac{t\Gamma^2}{\lambda}
          +M_r^r
           \left(\frac{t}{\lambda}\right)^{r-1}
      \right]\|f\|_1.
\end{align*}
Taking $t=\lambda/(M_r+\Gamma)$, we deduce that
$$
\lambda\varphi(1-e)
\leq C_{d,\beta,r}(M_r+\Gamma)\|f\|_1.
$$

For $r=\infty$, take
$$
t=\frac{\lambda}{2^{d+2}(M_\infty+\Gamma)}.
$$
Then, we have
$$
\sup_n\|\mathscr S_ng\|_\infty
\leq M_\infty\|g\|_\infty
\leq2^dtM_\infty
\leq\lambda/4.
$$
Thus by \eqref{eq:operator-good-bad} and
\eqref{eq:bad-output-bounds},
$$
\sup_n\|w\mathscr S_n(f)w\|_\infty\leq\lambda,
$$
and \eqref{eq:bad-projection-bound} yields
$$
\lambda\varphi(1-w)
\leq C_{d,\beta}(M_\infty+\Gamma)\|f\|_1.
$$

For arbitrary $f\in\mathscr D_0$, write
$$
f=(\operatorname{Re}f)_+
  -(\operatorname{Re}f)_-
  +i(\operatorname{Im}f)_+
  -i(\operatorname{Im}f)_-.
$$
The four positive parts belong to $\mathscr D_0$, and
their $L_1$ norms sum to at most $2\|f\|_1$.
Apply the positive estimate to each part at height
$\lambda/4$ and intersect the four projections.
The triangle inequality and trace subadditivity give
the conclusion.
\end{proof}

\subsection{Proof of the decomposition}
\label{subsec:dyadic-decomposition-proof}

Before proving Proposition~\ref{lemma:good-bad-decomposition},
we record two operator inequalities that will be used
to estimate the mixed terms.
For $x\in\cM_+$, $q\in\cP(\cM)$ and $t\geq0$,
\begin{equation}\label{eq:positive-compression-square}
    qxq\leq tq
    \quad\Longrightarrow\quad
    xqx\leq tx.
\end{equation}
Indeed,
$\|x^{1/2}qx^{1/2}\|_\infty
=\|qxq\|_\infty\leq t$, and hence
$$
xqx
=x^{1/2}(x^{1/2}qx^{1/2})x^{1/2}
\leq tx.
$$

\begin{lemma}\label{lemma:row-column-estimate}
Let $(a_k)_{k\geq1}\subset\cM$ and
$(b_k)_{k\geq1}\subset L_2(\cM)$ satisfy
$$
\sum_{k=1}^N a_ka_k^*\leq C1,
\qquad
\sum_{k\geq1}\|b_k\|_2^2<\infty,
$$
where $C\geq0$.
Let $z_{n,k}\in\cM$ satisfy
$\|z_{n,k}\|_\infty\leq1$ for all $n,k\geq1$.
Then
$$
x_n=\sum_{k\geq1}a_kz_{n,k}b_k
$$
converges in $L_2(\cM)$, and
\begin{equation}\label{eq:row-column-square}
    x_n^*x_n
    \leq C\sum_{k\geq1}b_k^*b_k
    \qquad(n\geq1).
\end{equation}
\end{lemma}

\begin{proof}
Let $(e_{ij})$ be the standard matrix units in $B(\ell_2)$.
For every finite set $F\subset\bN$,
$$
\left(\sum_{k\in F}a_kz_{n,k}b_k\right)\otimes e_{11}
=
\left(\sum_{k\in F}a_kz_{n,k}\otimes e_{1k}\right)
\left(\sum_{k\in F}b_k\otimes e_{k1}\right).
$$
Since $\|z_{n,k}\|_\infty\leq1$,
$$
\left\|\sum_{k\in F}a_kz_{n,k}\otimes e_{1k}\right\|_\infty^2
=
\left\|\sum_{k\in F}
a_kz_{n,k}z_{n,k}^*a_k^*\right\|_\infty
\leq
\left\|\sum_{k\in F}a_ka_k^*\right\|_\infty
\leq C.
$$
Consequently,
$$
\left|\sum_{k\in F}a_kz_{n,k}b_k\right|^2\otimes e_{11}
\leq
C\left|\sum_{k\in F}b_k\otimes e_{k1}\right|^2
=
C\left(\sum_{k\in F}b_k^*b_k\right)\otimes e_{11},
$$
and hence
$$
\left|\sum_{k\in F}a_kz_{n,k}b_k\right|^2
\leq C\sum_{k\in F}b_k^*b_k.
$$
Therefore we get
$$
\left\|\sum_{k\in F}a_kz_{n,k}b_k\right\|_2^2
\leq C\sum_{k\in F}\|b_k\|_2^2.
$$
Since $\sum_k\|b_k\|_2^2<\infty$, applying this estimate
to the tails proves that the series defining $x_n$
converges in $L_2(\cM)$.
Moreover, $\sum_kb_k^*b_k$ converges in $L_1(\cM)$ and
\( x_n^*x_n
    \leq C\sum_{k\geq1}b_k^*b_k.\)
\end{proof}

We now prove Proposition~\ref{lemma:good-bad-decomposition}.

\begin{proof}[Proof of Proposition~\ref{lemma:good-bad-decomposition}]
The case $f=0$ is immediate.
Write $f_j=\bE_jf$ and choose $j_0\in\bZ$ such that
$f_{j_0-1}\leq t1$. This is possible since
$$
\|\bE_jf\|_\infty
\leq2^{jd}|\supp_{\bR^d}f|\,\|f\|_\infty
\longrightarrow0
\qquad(j\to-\infty).
$$

We first construct the good function $g$ using the
Calder\'on--Zygmund decomposition in
Lemma~\ref{lemma:cuculescu-cz}, and verify the identities
for its dyadic averages needed below.
Let $q_j,p_j$ be the projections defined in
\eqref{eq:cuculescu-projections}.
By \eqref{eq:cuculescu-compressions},
\eqref{eq:cuculescu-trace} and
Lemma~\ref{lemma:cuculescu-cz},
\begin{equation}\label{eq:stopping-projection-bounds}
    \begin{gathered}
        q_jf_jq_j\leq tq_j,\qquad
        p_jf_jq_j=q_jf_jp_j=0,\\
        tp_j\leq p_jf_jp_j\leq2^dtp_j,\qquad
        t\sum_{j\geq j_0}\varphi(p_j)\leq\|f\|_1.
    \end{gathered}
\end{equation}

For $s\geq j_0$, put
$$
g_s=q_sf_sq_s+\sum_{j=j_0}^sp_jf_jp_j,
\qquad
g_{j_0-1}=f_{j_0-1}.
$$
Applying Lemma~\ref{lemma:cuculescu-cz} with $j_1=s$, we have
$$
0\leq g_s\leq2^dt1,\qquad
\|g_s\|_1=\|f\|_1,\qquad
\|g_s\|_2^2\leq2^dt\|f\|_1.
$$
Moreover, since $q_{s-1}=q_s+p_s$ and $p_sf_sq_s=0$, we get
$$
g_s=q_{s-1}f_sq_{s-1}
    +\sum_{j=j_0}^{s-1}p_jf_jp_j,
$$
and hence
$$
\bE_{s-1}g_s
=q_{s-1}f_{s-1}q_{s-1}
 +\sum_{j=j_0}^{s-1}p_jf_jp_j
=g_{s-1}.
$$
Thus $(g_s)$ is an $L_2$-bounded martingale and converges
in $L_2(\cN)$ to $g$ satisfying
$$
0\leq g\leq2^dt1,\qquad
\bE_sg=g_s
\quad(s\geq j_0-1).
$$
Moreover, we have
$$
\varphi(g)
=\varphi(\bE_{j_0-1}g)
=\varphi(f_{j_0-1})
=\|f\|_1.
$$
Note that all $g_s$ are supported on the finite
union of cubes in $\cQ_{j_0-1}$ meeting
$\supp_{\bR^d}f$, thus we have $g\in\mathscr D_0$,
and \eqref{eq:good-input-bounds} follows.

Write
$$
p_j=\sum_{P\in\cQ_j}\dsone_P\otimes p_P,
\qquad
f_Q=\bE_Qf.
$$
For $Q\in\cQ_s$, $s\geq j_0$, let $q_Q$ be the value
of $q_s$ on $Q$, and let $\mathcal P(Q)$ consist of
the ancestors of $Q$ at levels $j_0,\ldots,s$,
including $Q$.
The identity $\bE_sg=g_s$ yields
\begin{equation}\label{eq:good-cube-average}
    \bE_Qg
    =q_Qf_Qq_Q
     +\sum_{P\in\mathcal P(Q)}p_Pf_Pp_P.
\end{equation}
Consequently,
\begin{equation}\label{eq:dyadic-good-bad-decomposition}
    \begin{aligned}
        f_Q-\bE_Qg
        &=(1-q_Q)f_Q(1-q_Q)
          -\sum_{P\in\mathcal P(Q)}p_Pf_Pp_P\\
        &\quad +(1-q_Q)f_Qq_Q
          +q_Qf_Q(1-q_Q).
    \end{aligned}
\end{equation}

We next use the projection $\zeta$ defined in
\eqref{eq:enlarged-cuculescu-projection} to restrict
the terms involving $p_P$ to $(\beta P)^c$.
By Lemma~\ref{lemma:enlarged-cuculescu},
$$
\varphi(1-\zeta)\leq\frac{\beta^d}{t}\|f\|_1,
\qquad
\zeta(x)p_P=p_P\zeta(x)=0
\quad\text{for almost every }x\in\beta P.
$$
Put
$$
\Pi_Q(x)
=\sum_{P\in\mathcal P(Q)}
 \dsone_{(\beta P)^c}(x)p_P.
$$
The projections in this sum are pairwise orthogonal,
so $\Pi_Q$ is a projection.
Since $1-q_Q=\sum_{P\in\mathcal P(Q)}p_P$, the preceding
cancellation property gives
\begin{equation}\label{eq:stopping-localization}
    \zeta(1-q_Q)=\zeta\Pi_Q,
    \qquad
    (1-q_Q)\zeta=\Pi_Q\zeta.
\end{equation}
For any projection $w\leq\zeta$,
\eqref{eq:stopping-localization} gives
$w(1-q_Q)=w\Pi_Q$ and $(1-q_Q)w=\Pi_Qw$.
Therefore, \eqref{eq:dyadic-good-bad-decomposition} yields
\begin{align*}
    w(f_Q-\bE_Qg)w
    &=
    w\left(
        \Pi_Qf_Q\Pi_Q
        -\sum_{P\in\mathcal P(Q)}
         \dsone_{(\beta P)^c}p_Pf_Pp_P
    \right)w\\
    &\quad+w\Pi_Qf_Qq_Qw+wq_Qf_Q\Pi_Qw.
\end{align*}
Multiplying these three terms by $b_{n,Q}$ leads to
the series, we have
$$
B_n
=\sum_Qb_{n,Q}w
 \left(
     \Pi_Qf_Q\Pi_Q
     -\sum_{P\in\mathcal P(Q)}
      \dsone_{(\beta P)^c}p_Pf_Pp_P
 \right)w,
$$
and
$$
B_{\ell,n}=\sum_Qb_{n,Q}w\Pi_Qf_Qq_Qw,
\qquad
B_{r,n}=\sum_Qb_{n,Q}wq_Qf_Q\Pi_Qw.
$$
Here and below, $Q$ ranges over
$\bigcup_{s\geq j_0}\cQ_s$.
For $s<j_0$, we have
$$
\bE_sg
=\bE_s\bE_{j_0-1}g
=\bE_s\bE_{j_0-1}f
=\bE_sf.
$$
Thus $\bE_Q(f-g)=0$ for $Q\in\cQ_s$ with $s<j_0$.
These series give the decomposition
$$
w\mathscr S_n(f)w
=w\mathscr S_n(g)w+B_n+B_{\ell,n}+B_{r,n}.
$$

Fix $\lambda>0$. It remains to find a projection
$w\leq\zeta$ such that the series defining $B_n$
converges in $L_1(\cN)$, those defining $B_{\ell,n}$
and $B_{r,n}$ converge in $L_2(\cN)$, and
$$
\sup_n\max\{
    \|B_n\|_\infty,
    \|B_{\ell,n}\|_\infty,
    \|B_{r,n}\|_\infty
\}
\leq\frac{\lambda}{4},
$$
with
$$
\varphi(1-w)
\leq C_{d,\beta}
\left(
    \frac1t+\frac{\Gamma}{\lambda}
    +\frac{t\Gamma^2}{\lambda^2}
\right)\|f\|_1.
$$

We first estimate $\sum_Qm_Q\Pi_Q$.
Integrating \eqref{eq:exterior-budget} over $P$ gives
$$
\sum_{Q\subseteq P}\int_{(\beta P)^c}m_Q(x)\,dx
\leq\Gamma|P|.
$$
Thus, we obtain
\begin{align*}
    \varphi\left(\sum_Qm_Q\Pi_Q\right)
    &=
    \sum_{j\geq j_0}\sum_{P\in\cQ_j}
    \tau(p_P)
    \sum_{Q\subseteq P}\int_{(\beta P)^c}m_Q(x)\,dx\\
    &\leq
    \Gamma\sum_{j\geq j_0}\sum_{P\in\cQ_j}
    |P|\tau(p_P)\\
    &=\Gamma\sum_{j\geq j_0}\varphi(p_j)
    \leq\frac{\Gamma}{t}\|f\|_1,
\end{align*}
where the last inequality follows from
\eqref{eq:cuculescu-trace}.
Set
$$
w_0
=\zeta\wedge
 \dsone_{[0,\Gamma]}
 \left(\sum_Qm_Q\Pi_Q\right).
$$
For $\Gamma>0$, Chebyshev's inequality yields
$$
\varphi(1-w_0)
\leq\varphi(1-\zeta)
+\frac1\Gamma\varphi\left(\sum_Qm_Q\Pi_Q\right)
\leq\frac{\beta^d+1}{t}\|f\|_1.
$$
If $\Gamma=0$, the positive series is zero and
$w_0=\zeta$. Hence, in both cases, we have
\begin{equation}\label{eq:common-input-projection}
    w_0\left(\sum_Qm_Q\Pi_Q\right)w_0\leq\Gamma w_0,
    \qquad
    \varphi(1-w_0)\leq\frac{\beta^d+1}{t}\|f\|_1.
\end{equation}

We next estimate the positive terms appearing in $B_n$.
By orthogonality of the projections indexed by
$\mathcal P(Q)$, we get
$$
\tau(\Pi_Q(x)f_Q\Pi_Q(x))
=
\sum_{P\in\mathcal P(Q)}
\dsone_{(\beta P)^c}(x)\tau(p_Pf_Q).
$$
Since
$$
\tau(p_Pf_Q)
=\frac1{|Q|}\int_Q\tau(p_Pf(y))\,dy
\qquad(Q\subseteq P),
$$
we obtain
\begin{equation}\label{eq:first-positive-trace}
    \begin{aligned}
        \sum_Q\varphi(m_Q\Pi_Qf_Q\Pi_Q)
        &=
        \sum_{j\geq j_0}\sum_{P\in\cQ_j}
        \int_P\tau(p_Pf(y))
        \left(
            \sum_{\substack{Q\subseteq P\\y\in Q}}
            \frac1{|Q|}
            \int_{(\beta P)^c}m_Q(x)\,dx
        \right)dy\\
        &\leq
        \Gamma\sum_{j\geq j_0}\varphi(p_jf)
        \leq\Gamma\|f\|_1,
    \end{aligned}
\end{equation}
where we used \eqref{eq:exterior-budget} and
$\sum_{j\geq j_0}p_j\leq1$.
Similarly,
\begin{equation}\label{eq:second-positive-trace}
    \sum_Q
    \varphi\left(
        m_Q\sum_{P\in\mathcal P(Q)}
        \dsone_{(\beta P)^c}p_Pf_Pp_P
    \right)
    \leq\Gamma\sum_{j\geq j_0}\varphi(p_jf)
    \leq\Gamma\|f\|_1.
\end{equation}
On the other hand, $q_Qf_Qq_Q\leq tq_Q$ and
\eqref{eq:positive-compression-square} give
$f_Qq_Qf_Q\leq tf_Q$. Consequently,
\begin{equation}\label{eq:mixed-L2-sum}
    \begin{aligned}
        \sum_Q\|\sqrt{m_Q}\Pi_Qf_Qq_Qw_0\|_2^2
        \leq
        \sum_Q\varphi(m_Q\Pi_Qf_Qq_Qf_Q\Pi_Q)
        \leq t\Gamma\|f\|_1.
    \end{aligned}
\end{equation}
We now define $h\in L_1(\cN)_+$ by
$$
\begin{aligned}
    h
    =\frac4\lambda
      \sum_Qm_Qw_0
      \left(
          \Pi_Qf_Q\Pi_Q
          +\sum_{P\in\mathcal P(Q)}
           \dsone_{(\beta P)^c}p_Pf_Pp_P
      \right)w_0+
      \frac{16\Gamma}{\lambda^2}
      \sum_Qm_Qw_0q_Qf_Q\Pi_Qf_Qq_Qw_0.
\end{aligned}
$$
By \eqref{eq:first-positive-trace},
\eqref{eq:second-positive-trace} and
\eqref{eq:mixed-L2-sum}, we get
$$
\varphi(h)
\leq
\left(
    \frac{8\Gamma}{\lambda}
    +\frac{16t\Gamma^2}{\lambda^2}
\right)\|f\|_1.
$$
Take
$$
w=w_0\wedge\dsone_{[0,1]}(h).
$$
Then $w\leq\zeta$, $whw\leq w$, and
\begin{align*}
    \varphi(1-w)
    \leq\varphi(1-w_0)+\varphi(h)\leq
    \left(
        \frac{\beta^d+1}{t}
        +\frac{8\Gamma}{\lambda}
        +\frac{16t\Gamma^2}{\lambda^2}
    \right)\|f\|_1.
\end{align*}
This proves \eqref{eq:bad-projection-bound}.

Since $|b_{n,Q}|\leq m_Q$,
\eqref{eq:first-positive-trace} and
\eqref{eq:second-positive-trace} give absolute convergence
of the series defining $B_n$ in $L_1(\cN)$.
For $Y\geq0$ and $|z|\leq1$,
$$
\begin{pmatrix}
    Y&zY\\
    \overline zY&Y
\end{pmatrix}\geq0.
$$
Applying this to the two positive terms in $B_n$ with
coefficients $b_{n,Q}/m_Q$ and $-b_{n,Q}/m_Q$, and
summing, gives
\begin{equation}\label{eq:error-diagonal-block}
    \begin{pmatrix}
        \dfrac{\lambda}{4}whw&B_n\\
        B_n^*&\dfrac{\lambda}{4}whw
    \end{pmatrix}\geq0.
\end{equation}
As $whw\leq w$, Lemma~\ref{lemma:positive-block} yields
$\|B_n\|_\infty\leq\lambda/4$.

Furthermore, by \eqref{eq:common-input-projection},
\eqref{eq:mixed-L2-sum} and the fact $w\leq w_0$, we obtain
$$
\sum_Q
(\sqrt{m_Q}\,w\Pi_Q)
(\sqrt{m_Q}\,w\Pi_Q)^*
\leq\Gamma w,
\qquad
\sum_Q\|\sqrt{m_Q}\Pi_Qf_Qq_Qw\|_2^2
\leq t\Gamma\|f\|_1.
$$
Apply Lemma~\ref{lemma:row-column-estimate} with
$$
a_Q=\sqrt{m_Q}\,w\Pi_Q,\qquad
b_Q=\sqrt{m_Q}\,\Pi_Qf_Qq_Qw,\qquad
z_{n,Q}=b_{n,Q}/m_Q.
$$
Applying the same lemma with $\overline{z_{n,Q}}$ and
taking adjoints proves that the series defining
$B_{\ell,n}$ and $B_{r,n}$ both converge in $L_2(\cN)$ and satisfy
\begin{equation}\label{eq:common-column-square}
    \begin{aligned}
        B_{\ell,n}^*B_{\ell,n},\ B_{r,n}B_{r,n}^*
        \leq
        \Gamma\sum_Qm_Qwq_Qf_Q\Pi_Qf_Qq_Qw
        \leq\frac{\lambda^2}{16}whw.
    \end{aligned}
\end{equation}
Thus
$\|B_{\ell,n}\|_\infty,\|B_{r,n}\|_\infty
\leq\lambda/4$, proving \eqref{eq:bad-output-bounds}.
\end{proof}

\section{A simultaneous dyadic decomposition}
\label{sec:kernel-decomposition}

In this section, we construct a simultaneous dyadic
decomposition of a family of kernels;
see Proposition~\ref{prop:dyadic-majorants}.
We then express the convolution operators in terms
of dyadic averages and verify the coefficient bound
required in Theorem~\ref{prop:exterior-endpoint};
see Proposition~\ref{prop:kernel-pairing}.

\subsection{Dyadic expansion and coefficient estimates}
\label{subsec:kernel-decomposition}

In this subsection, we construct a dyadic expansion for a countable
family $(K_n)_{n\geq1}\subset L_{1,\mathrm{loc}}(\bR^d\setminus\{0\})$
of complex-valued kernels.

We use the dyadic cubes and averages introduced in
Section~\ref{subsec:cuculescu}.
For each $j\in\bZ$, let
\[
    V_j=[-2^{1-j},2^{1-j})^d
        \setminus[-2^{-j},2^{-j})^d,
    \qquad
    \mathcal B_j=\{B\in\cQ_j:B\subseteq V_j\}.
\]
The $4^d-2^d$ cubes in $\mathcal B_j$ partition $V_j$,
and the cubes in $\bigcup_j\mathcal B_j$ partition
$\bR^d\setminus\{0\}$.
Let $\cD$ consist of these cubes and all their dyadic
descendants. Each $Q\in\cD$ is therefore contained in
a unique cube $B\in\bigcup_j\mathcal B_j$.

On each $B\in\bigcup_j\mathcal B_j$, we define the
coefficients $a_{n,Q}$ recursively, starting with $B$
and proceeding to its dyadic descendants.
Once the coefficients on the strict ancestors of $Q$
have been chosen, put
\[
    F_{n,Q}
    =K_n-\sum_{Q\subsetneq P\subseteq B}a_{n,P},
\]
where the sum is empty for $Q=B$.
Set
\[
    d_{n,Q}=
    \begin{cases}
       \bE_QF_{n,Q},
        &\displaystyle
         \int_Q|F_{n,Q}-\bE_QF_{n,Q}|
         \leq\frac12\int_Q|F_{n,Q}|,\\[4pt]
        0,&\text{otherwise},
    \end{cases}
\]
and define
\begin{equation}\label{eq:mean-selection}
    (a_{n,Q})_n=
    \begin{cases}
        (d_{n,Q})_n,
        &\displaystyle
         \int_Q\sup_n|F_{n,Q}-d_{n,Q}|
         \leq\frac12\int_Q\sup_n|F_{n,Q}|,\\[4pt]
        (0)_n,&\text{otherwise}.
    \end{cases}
\end{equation}

We impose the following assumptions on $(K_n)_{n\geq1}$.
The kernels are uniformly bounded and Lipschitz, in the sense that
$$
\sup_n\|K_n\|_\infty
+\sup_n\sup_{x\ne y}
\frac{|K_n(x)-K_n(y)|}{|x-y|}<\infty.
$$
They satisfy the joint H\"ormander condition
\eqref{eq:dyadic-hormander},
$$
C_D:=\sup_{k\in\bZ}\sup_{y\in Q_0^{(k)}}
\int_{(AQ_0^{(k)})^c}
\sup_n|K_n(x-y)-K_n(x)|\,dx<\infty
$$
for some fixed $A\geq1$. In addition, we require a uniform
bound on the integrals of $\sup_n|K_n|$ over $V_j$:
$$
C_V:=\sup_{j\in\bZ}\int_{V_j}\sup_n|K_n(x)|\,dx<\infty.
$$

\begin{proposition}\label{prop:dyadic-majorants}
The coefficients defined above have the following properties.

\smallskip
\noindent\textup{(i)}
For every $B\in\bigcup_j\mathcal B_j$,
\begin{equation}\label{eq:local-coefficient-mass}
    \sum_{Q\subseteq B}|Q||a_{n,Q}|
    \leq2\int_B|K_n|,
    \qquad
    \sum_{Q\subseteq B}|Q|\sup_n|a_{n,Q}|
    \leq2\int_B\sup_n|K_n|.
\end{equation}
Furthermore, for every $P\in\cD$,
\begin{equation}\label{eq:tree-tail}
    \sum_{Q\subsetneq P}|Q|\sup_n|a_{n,Q}|
    \leq12\int_P\sup_n|K_n-\bE_PK_n|.
\end{equation}

\smallskip
\noindent\textup{(ii)}
For every $n\geq1$,
\begin{equation}\label{eq:exact-kernel-series}
    K_n=\sum_{Q\in\cD}a_{n,Q}\dsone_Q
    \quad\text{almost everywhere on }\bR^d\setminus\{0\}.
\end{equation}
 If $K_n\in L_1(\bR^d)$, then
\begin{equation}\label{eq:dyadic-majorants}
    |K_n|\leq\sum_{Q\in\cD}|a_{n,Q}|\dsone_Q
    \quad\text{a.e.},
    \qquad
    \sum_{Q\in\cD}|Q||a_{n,Q}|\leq2\|K_n\|_1.
\end{equation}

\smallskip
\noindent\textup{(iii)}
For $Q\in\cQ$, put $\rho_Q=|c_Q|_\infty+\ell(Q)/2$, and define
$$
\mathcal E_j
=\{Q\in\cQ:\ell(Q)\leq2^{-j},\ \rho_Q>2^{-j}\}.
$$
Then, for every $j\in\bZ$,
\begin{equation}\label{eq:dyadic-displacement}
    \sum_{Q\in\cD\cap\mathcal E_j}|Q|\sup_n|a_{n,Q}|
    \leq C_{d,A}(C_V+C_D).
\end{equation}
\end{proposition}

\begin{proof}
We first obtain the coefficient estimates needed to justify
the expansion. By the definition of $a_{n,Q}$,
\begin{align}
    |Q||a_{n,Q}|
    &\leq2\left(
        \int_Q|F_{n,Q}|
        -\int_Q|F_{n,Q}-a_{n,Q}|
    \right),\label{eq:individual-mass-decrease}\\
    |Q|\sup_n|a_{n,Q}|
    &\leq2\left(
        \int_Q\sup_n|F_{n,Q}|
        -\int_Q\sup_n|F_{n,Q}-a_{n,Q}|
    \right).\label{eq:common-mass-decrease}
\end{align}
Fix $B\in\bigcup_j\mathcal B_j$, and denote its $k$th generation
of dyadic subcubes by
$$
\mathcal G_k(B)
=\{Q\subseteq B:\ell(Q)=2^{-k}\ell(B)\}.
$$
For each child $Q'$ of $Q$, the recursive definition gives
$F_{n,Q'}=F_{n,Q}-a_{n,Q}$. Indeed,
$$
F_{n,Q'}
=K_n-\sum_{Q'\subsetneq P\subseteq B}a_{n,P}
=K_n-a_{n,Q}-\sum_{Q\subsetneq P\subseteq B}a_{n,P}
=F_{n,Q}-a_{n,Q}.
$$
Consequently,
$$
\sum_{\substack{Q'\text{ child}\\\text{of }Q}}
\int_{Q'}\sup_n|F_{n,Q'}|
=\int_Q\sup_n|F_{n,Q}-a_{n,Q}|.
$$
Summing \eqref{eq:common-mass-decrease} and \eqref{eq:individual-mass-decrease} over generations
$0,\ldots,m$, all intermediate integrals cancel, and we obtain
\begin{align*}
    \sum_{k=0}^m\sum_{Q\in\mathcal G_k(B)}
    |Q|\sup_n|a_{n,Q}|
    &\leq2\left(
        \int_B\sup_n|K_n|
        -\sum_{Q\in\mathcal G_m(B)}
         \int_Q\sup_n|F_{n,Q}-a_{n,Q}|
    \right)\\
    &\leq2\int_B\sup_n|K_n|,
\end{align*}
and
\[\sum_{k=0}^m\sum_{Q\in\mathcal G_k(B)}
    |Q||a_{n,Q}|
  \leq2\int_B|K_n|.\]
Letting $m\to\infty$ proves the estimates in
\eqref{eq:local-coefficient-mass}. 
Repeating the same summation over the strict descendants of $P$
and noting that $F_{n,Q}=F_{n,P}-a_{n,P}$ for
$Q\in\mathcal G_1(P)$, we obtain
\begin{equation}\label{eq:descendant-residual}
    \sum_{Q\subsetneq P}|Q|\sup_n|a_{n,Q}|
    \leq2\int_P\sup_n|F_{n,P}-a_{n,P}|.
\end{equation}
We claim that
\begin{equation}\label{eq:outgoing-oscillation}
    \int_P\sup_n|F_{n,P}-a_{n,P}|
    \leq6\int_P\sup_n|K_n-\bE_PK_n|.
\end{equation}
Together with \eqref{eq:descendant-residual}, this yields
$$
\sum_{Q\subsetneq P}|Q|\sup_n|a_{n,Q}|
\leq12\int_P\sup_n|K_n-\bE_PK_n|,
$$
which is \eqref{eq:tree-tail}.

To prove the claim, note that
$
F_{n,P}-\bE_PF_{n,P}
=K_n-\bE_PK_n.$
If $d_{n,P}=\bE_PF_{n,P}$, then
$|F_{n,P}-d_{n,P}|=|K_n-\bE_PK_n|$.
Otherwise, $d_{n,P}=0$ and
$$
\int_P|F_{n,P}|
<2\int_P|F_{n,P}-\bE_PF_{n,P}|
=2\int_P|K_n-\bE_PK_n|.
$$
Consequently,
$$
|\bE_PF_{n,P}|
\leq\bE_P|F_{n,P}|
\leq2\bE_P|K_n-\bE_PK_n|,
$$
so in both cases,
$$
|F_{n,P}-d_{n,P}|
\leq|K_n-\bE_PK_n|+2\bE_P|K_n-\bE_PK_n|.
$$
Taking the supremum over $n$ and integrating over $P$ gives
$$
\int_P\sup_n|F_{n,P}-d_{n,P}|
\leq3\int_P\sup_n|K_n-\bE_PK_n|.
$$
By \eqref{eq:mean-selection}, either $a_{n,P}=d_{n,P}$
for all $n$, or $a_{n,P}=0$ for all $n$ and
$$
\int_P\sup_n|F_{n,P}|
<2\int_P\sup_n|F_{n,P}-d_{n,P}|.
$$
Thus
$$
\int_P\sup_n|F_{n,P}-a_{n,P}|
\leq2\int_P\sup_n|F_{n,P}-d_{n,P}|
\leq6\int_P\sup_n|K_n-\bE_PK_n|,
$$
proving the claim.

We now prove \eqref{eq:exact-kernel-series}. Fix
$B\in\bigcup_j\mathcal B_j$. For $P\in\mathcal G_m(B)$
and $x\in P$,
$$
K_n(x)-\sum_{k=0}^m\sum_{Q\in\mathcal G_k(B)}
a_{n,Q}\dsone_Q(x)
=K_n(x)-\sum_{P\subseteq Q\subseteq B}a_{n,Q}
=F_{n,P}(x)-a_{n,P}.
$$
Let $C_{\mathrm{Lip}}$ be a common Lipschitz constant. Since
$$
|K_n(x)-\bE_PK_n|
\leq\frac1{|P|}\int_P|K_n(x)-K_n(y)|\,dy
\leq C_{\mathrm{Lip}}\sqrt d\,\ell(P),
\qquad x\in P,
$$
we obtain from \eqref{eq:outgoing-oscillation} that
\begin{align*}
 \int_B\sup_n\left|K_n-\sum_{k=0}^m\sum_{Q\in\mathcal G_k(B)}
    a_{n,Q}\dsone_Q\right|
    &=\sum_{P\in\mathcal G_m(B)}
    \int_P\sup_n|F_{n,P}-a_{n,P}|\\
    &\leq6\sum_{P\in\mathcal G_m(B)}
    \int_P\sup_n|K_n-\bE_PK_n|\\
    &\leq6C_{\mathrm{Lip}}\sqrt d\,
    2^{-m}\ell(B)|B|\longrightarrow0.
\end{align*}
Moreover, by \eqref{eq:local-coefficient-mass}, we have
$$
\int_B\sum_{Q\subseteq B}\sup_n|a_{n,Q}|\dsone_Q
=\sum_{Q\subseteq B}|Q|\sup_n|a_{n,Q}|
\leq2\int_B\sup_n|K_n|<\infty.
$$
Thus the series converge absolutely almost everywhere and in
$L_1(B)$, and
$$
\left\|K_n-\sum_{Q\subseteq B}a_{n,Q}\dsone_Q\right\|_{L_1(B)}
=\lim_{m\to\infty}
\left\|K_n-\sum_{k=0}^m\sum_{Q\in\mathcal G_k(B)}
a_{n,Q}\dsone_Q\right\|_{L_1(B)}
=0.
$$
Since $\bigcup_j\mathcal B_j$ is a countable partition of
$\bR^d\setminus\{0\}$, this proves \eqref{eq:exact-kernel-series}.

If $K_n\in L_1(\bR^d)$, then
$$
|K_n|\leq\sum_{Q\in\cD}|a_{n,Q}|\dsone_Q
\quad\text{a.e.},
$$
and summing \eqref{eq:local-coefficient-mass} over $B$,
$$
\sum_{Q\in\cD}|Q||a_{n,Q}|
=\sum_{B\in\bigcup_j\mathcal B_j}
\sum_{Q\subseteq B}|Q||a_{n,Q}|
\leq2\sum_{B\in\bigcup_j\mathcal B_j}\int_B|K_n|
=2\|K_n\|_1.
$$
This proves \eqref{eq:dyadic-majorants}.

It remains to prove \eqref{eq:dyadic-displacement}.
Fix $j\in\bZ$ and choose $J=\lceil\log_2(2A)\rceil$.
Every $Q\in\cD\cap\mathcal E_j$ is contained in a cube
$B\in\mathcal B_r$ with $r\leq j$. Indeed, if $r>j$, then
$B\subseteq[-2^{-j},2^{-j})^d$, contrary to $\rho_Q>2^{-j}$.
For $j-J<r\leq j$, by \eqref{eq:local-coefficient-mass}, we have
$$
\sum_{r=j-J+1}^j\sum_{B\in\mathcal B_r}
\sum_{Q\subseteq B}|Q|\sup_n|a_{n,Q}|
\leq2\sum_{r=j-J+1}^j\int_{V_r}\sup_n|K_n|
\leq2JC_V.
$$
For $r\leq j-J$ and $B\in\mathcal B_r$, write
$$
B=\bigsqcup_{\substack{P\in\cQ_{j-1}\\P\subseteq B}}P.
$$
Since $2^{-r}\geq2^{J-j}\geq A2^{1-j}$, each such $P$ satisfies
$$
P\subseteq B\subseteq V_r
\subseteq\bigl([-2^{-r},2^{-r})^d\bigr)^c
\subseteq(AQ_0^{(j-2)})^c.
$$
If $Q\in\cD\cap\mathcal E_j$ and $Q\subseteq B$, then
$\ell(Q)\leq2^{-j}<2^{1-j}=\ell(P)$.
By dyadic nesting, $Q\subsetneq P$ for a unique $P$ in
this partition. Summing over all such $P$ and applying
\eqref{eq:tree-tail}, we obtain
\begin{align*}
    \sum_{r\leq j-J}\sum_{B\in\mathcal B_r}
    \sum_{\substack{Q\in\cD\cap\mathcal E_j\\Q\subseteq B}}
    |Q|\sup_n|a_{n,Q}|
    \leq\sum_P\sum_{Q\subsetneq P}|Q|\sup_n|a_{n,Q}|
    \leq12\sum_P\int_P\sup_n|K_n-\bE_PK_n|.
\end{align*}

Write $D(z,y)=\sup_n|K_n(z-y)-K_n(z)|$.
We have
$$
\sup_n|K_n(z)-\bE_PK_n|
\leq\frac1{|P|}\int_PD(z,z-w)\,dw,
\qquad z\in P.
$$
Since $|P|=2^{(1-j)d}$, changing variables $y=z-w$ yields
\begin{align*}
    \sum_P\int_P\sup_n|K_n-\bE_PK_n|
    &\leq2^{(j-1)d}
      \sum_P\int_P\int_PD(z,z-w)\,dw\,dz\\
    &=2^{(j-1)d}
      \int_{(-2^{1-j},2^{1-j})^d}\int_{\bR^d}
      D(z,y)\sum_P\dsone_P(z)\dsone_P(z-y)\,dz\,dy\\
    &\leq2^{(j-1)d}
      \int_{(-2^{1-j},2^{1-j})^d}
      \int_{(AQ_0^{(j-2)})^c}D(z,y)\,dz\,dy\\
    &\leq2^dC_D.
\end{align*}
Here disjointness gives
$\sum_P\dsone_P(z)\dsone_P(z-y)\leq1$,
and the last inequality follows from
\eqref{eq:dyadic-hormander}, since
$(-2^{1-j},2^{1-j})^d\subset Q_0^{(j-2)}$.
Combining the two contributions, we get
$$
\sum_{Q\in\cD\cap\mathcal E_j}|Q|\sup_n|a_{n,Q}|
\leq2JC_V+12\cdot2^dC_D
\leq C_{d,A}(C_V+C_D),
$$
which proves \eqref{eq:dyadic-displacement}.
\end{proof}

\begin{remark}\label{remark:dyadic-majorants}
The conclusions of Proposition~\ref{prop:dyadic-majorants}
also hold for a finite family
$(K_n)_{n=1}^N\subset L_{1,\mathrm{loc}}(\bR^d\setminus\{0\})$,
under the same assumptions involving $C_D$ and $C_V$, without
boundedness or Lipschitz regularity. Indeed, on each
$B\in\bigcup_j\mathcal B_j$,
$$
\int_B\max_{1\leq n\leq N}|K_n|
\leq\sum_{n=1}^N\int_B|K_n|<\infty,
$$
so the recursion and the coefficient estimates remain valid.
For $B\in\mathcal B_j$, by \eqref{eq:outgoing-oscillation},
\begin{align*}
    \int_B\max_{1\leq n\leq N}
    \left|K_n-\sum_{\ell=0}^k
    \sum_{Q\in\mathcal G_\ell(B)}a_{n,Q}\dsone_Q\right|\leq
    6\sum_{n=1}^N
    \|K_n-\bE_{j+k}(\dsone_BK_n)\|_{L_1(B)}
    \longrightarrow0.
\end{align*}

The identity
$$
a_{n,Q}\dsone_Q
=\bigl(|Q|a_{n,Q}\bigr)\frac{\dsone_Q}{|Q|}
$$
explains the appearance of $|Q|a_{n,Q}$ when the expansion
is written using normalized cube indicators.
The estimates hold for the same coefficients at every scale.
Global integrability is needed only for the global coefficient
bound in \eqref{eq:dyadic-majorants}; the other conclusions
allow $\sum_Q|Q||a_{n,Q}|=\infty$.

Related dyadic estimates in terms of local oscillation appear
in Lerner~\cite{lerner10pointwise} and in the Banach space valued
decomposition of H\"anninen--Hyt\"onen
\cite[Theorem~2.16]{hanninenhytonen14oscillation}.
Here the two selection conditions give both the individual
estimate in \eqref{eq:local-coefficient-mass} and the common
estimate \eqref{eq:tree-tail}, while retaining the exact
expansion for complex-valued kernels.

Finally, if $\sup_n\|K_n\|_1\leq A_1$, then the assumption
$C_V<\infty$ follows from \eqref{eq:dyadic-hormander}.
To see this, let $Q_a=[-a/2,a/2)^d$ have dyadic side length $a$.
Averaging over $y\in Q_a$ gives
$$
|K_n(z)|
\leq a^{-d}A_1
+\frac1{|Q_a|}\int_{Q_a}
\sup_m|K_m(z-y)-K_m(z)|\,dy.
$$
Taking the supremum over $n$ and integrating over a measurable
set $U\subseteq(AQ_a)^c$, we obtain
$$
\int_U\sup_n|K_n|
\leq |U|a^{-d}A_1+C_D.
$$
For each $j$, choose $a$ such that
$2^{-j}/A<a\leq2^{1-j}/A$ and take $U=V_j$.
Then $V_j\subseteq(AQ_a)^c$ and $|V_j|a^{-d}\leq C_{d,A}$,
so
$
C_V\leq C_{d,A}A_1+C_D.
$
Thus \eqref{eq:dyadic-majorants} and
\eqref{eq:dyadic-displacement} give the two coefficient
bounds required for integrable kernels.
\end{remark}

\subsection{Dyadic representations of convolution operators}
\label{subsec:kernel-pairing}

We use the kernel expansion from
Proposition~\ref{prop:dyadic-majorants} to express $K_n*f$
in terms of the dyadic averages $\bE_Qf$.
We first establish the corresponding identity for convolution
with a cube indicator.

Write $\widehat Q$ for the dyadic parent of $Q$.
For an axis-parallel cube $R$ of dyadic side length and
$Q\in\cQ$, put
$$
X_Q^R
=\{x\in\bR^d:Q\subseteq x-R,\
                 \widehat Q\not\subseteq x-R\}.
$$
Thus $x\in X_Q^R$ precisely when $Q$ is a maximal standard
dyadic cube contained in $x-R$.

For each $x$, these maximal cubes are pairwise disjoint
and cover the interior of $x-R$ up to a null set.
Indeed, every interior point outside the dyadic boundaries
belongs to a sufficiently small dyadic cube contained in
$x-R$, and only finitely many of its ancestors can be
contained in this bounded set. Consequently,
\begin{equation}\label{eq:exact-cube-pairing}
    \dsone_R(x-y)
    =\sum_{Q\in\cQ}
      \dsone_{X_Q^R}(x)\dsone_Q(y)
    \quad\text{for almost every }(x,y).
\end{equation}
For any Bochner integrable function $f$, we have
\begin{equation}\label{eq:cube-average-representation}
    (\dsone_R*f)(x)
    =\sum_{Q\in\cQ}
      |Q|\dsone_{X_Q^R}(x)\bE_Qf
    \quad\text{for almost every }x.
\end{equation}

If $\ell(Q)\leq\ell(R)$, the sets of translations containing
$Q$ and $\widehat Q$ have measures
$(\ell(R)-\ell(Q))^d$ and
$(\ell(R)-2\ell(Q))_+^d$, respectively. Hence
\begin{equation}\label{eq:cube-pairing-volume}
    \begin{aligned}
        |X_Q^R|
        =(\ell(R)-\ell(Q))^d
          -(\ell(R)-2\ell(Q))_+^d\leq d\,\ell(Q)\ell(R)^{d-1}.
    \end{aligned}
\end{equation}
If $\ell(Q)>\ell(R)$, then $X_Q^R=\varnothing$.

Let $(K_n)_{n\geq1}$ satisfy the hypotheses of
Proposition~\ref{prop:dyadic-majorants}, and let
$(a_{n,R})$ be the coefficients constructed there.
For $Q\in\cQ$, define
$$
b_{n,Q}(x)
=|Q|\sum_{R\in\cD}a_{n,R}\dsone_{X_Q^R}(x),
\qquad
m_Q(x)
=|Q|\sum_{R\in\cD}\sup_n|a_{n,R}|
                         \dsone_{X_Q^R}(x).
$$
\begin{lemma}\label{lemma:pairing-coefficients}
For every $Q\in\cQ$, the series defining $b_{n,Q}$ and $m_Q$
converge absolutely in $L_1(\bR^d)$ and uniformly on $\bR^d$.
Moreover,
\begin{equation}\label{eq:pairing-coefficient-bounds}
    |b_{n,Q}|\leq m_Q,\qquad
    \|m_Q\|_\infty\leq C_{d,A}(C_V+C_D),\qquad
    \|m_Q\|_1\leq C_{d,A}(C_V+C_D)|Q|.
\end{equation}
\end{lemma}

\begin{proof}
Every $R\in\cD$ satisfies $\rho_R>\ell(R)$.
By \eqref{eq:dyadic-displacement}, we have
\begin{equation}\label{eq:pairing-scale-bound}
    \sum_{\substack{R\in\cD\\\ell(R)=2^{-s}}}
    |R|\sup_n|a_{n,R}|
    \leq
    \sum_{R\in\cD\cap\mathcal E_s}
    |R|\sup_n|a_{n,R}|
    \leq C_{d,A}(C_V+C_D).
\end{equation}
Fix $Q\in\cQ_q$. Since $X_Q^R=\varnothing$ whenever
$\ell(R)<\ell(Q)$,
\begin{align*}
    \sum_{R\in\cD}
    \left\|
        |Q|\sup_n|a_{n,R}|\dsone_{X_Q^R}
    \right\|_\infty
    &\leq
    \sum_{s\leq q}2^{-d(q-s)}
    \sum_{\substack{R\in\cD\\\ell(R)=2^{-s}}}
    |R|\sup_n|a_{n,R}|\\
    &\leq
    C_{d,A}(C_V+C_D)\sum_{s\leq q}2^{-d(q-s)}
    \leq C_{d,A}(C_V+C_D).
\end{align*}
Similarly, by \eqref{eq:cube-pairing-volume},
\begin{align*}
    \sum_{R\in\cD}
    \left\|
        |Q|\sup_n|a_{n,R}|\dsone_{X_Q^R}
    \right\|_1
    &\leq
    d|Q|\sum_{s\leq q}2^{-(q-s)}
    \sum_{\substack{R\in\cD\\\ell(R)=2^{-s}}}
    |R|\sup_n|a_{n,R}|\\
    &
    \leq C_{d,A}(C_V+C_D)|Q|.
\end{align*}
These estimates prove absolute convergence in $L_1$ and
uniform convergence of the series defining $m_Q$, together
with its asserted norm bounds. Since
$|a_{n,R}|\leq\sup_m|a_{m,R}|$, we also have
convergence for $b_{n,Q}$ and the inequality $|b_{n,Q}|\leq m_Q$.
\end{proof}

The next proposition expresses convolution in terms of
dyadic averages and controls the coefficients outside
enlarged dyadic cubes.

\begin{proposition}\label{prop:kernel-pairing}
Let $\mathcal X$ be a Banach space and
$f\in L_1(\bR^d;\mathcal X)$. Then, for every $n\geq1$,
\begin{equation}\label{eq:input-cube-expansion}
    (K_n*f)(x)
    =\sum_{Q\in\cQ}b_{n,Q}(x)\bE_Qf
    \quad\text{for almost every }x\in\bR^d.
\end{equation}
For every bounded measurable set $E\subset\bR^d$,
\begin{equation}\label{eq:input-cube-majorant}
    \int_E\sum_{Q\in\cQ}
    m_Q(x)\|\bE_Qf\|_{\mathcal X}\,dx<\infty.
\end{equation}
In particular, the series in
\eqref{eq:input-cube-expansion} converges absolutely
almost everywhere and in $L_1(E;\mathcal X)$.

Moreover, for every $P\in\cQ$ and $\beta\geq3$,
\begin{equation}\label{eq:pairing-exterior-bound}
    \sum_{\substack{Q\subseteq P\\y\in Q}}
    \frac1{|Q|}
    \int_{(\beta P)^c}m_Q(x)\,dx
    \leq C_{d,A}(C_V+C_D)
    \quad\text{for almost every }y\in P.
\end{equation}
\end{proposition}

\begin{proof}
Since $V_j$ is the disjoint union of the cubes in
$\mathcal B_j$, and every $R\in\cD$ is contained in a
unique $B\in\bigcup_k\mathcal B_k$,
\begin{align*}
    \int_{V_j}\sum_{R\in\cD}
    \sup_n|a_{n,R}|\dsone_R(x)\,dx
    =\sum_{R\in\cD}\sup_n|a_{n,R}|\,|R\cap V_j|
    =\sum_{B\in\mathcal B_j}
      \sum_{\substack{R\in\cD\\R\subseteq B}}
      |R|\sup_n|a_{n,R}|.
\end{align*}
Applying \eqref{eq:local-coefficient-mass} on each
$B\in\mathcal B_j$, we obtain
\begin{align*}
    \int_{V_j}\sum_{R\in\cD}
    \sup_n|a_{n,R}|\dsone_R(x)\,dx
    \leq2\sum_{B\in\mathcal B_j}
      \int_B\sup_n|K_n(x)|\,dx
     \leq2C_V.
\end{align*}
For $M>0$, every $B\in\bigcup_j\mathcal B_j$ satisfying
$|B\cap[-M,M]^d|>0$ is contained in $[-2M,2M]^d$.
Since these cubes are pairwise disjoint, by
\eqref{eq:local-coefficient-mass}, we have
$$
\int_{[-M,M]^d}\sum_{R\in\cD}
\sup_n|a_{n,R}|\dsone_R(x)\,dx
\leq2\int_{[-2M,2M]^d}\sup_n|K_n|
\leq2(4M)^d\sup_n\|K_n\|_\infty.
$$

Let $E\subseteq[-M,M]^d$ be measurable.
If $|y|_\infty\leq2M$, then
$E-y\subseteq[-3M,3M]^d$, and hence
$$
\int_E\sum_{R\in\cD}
\sup_n|a_{n,R}|\dsone_R(x-y)\,dx
\leq2(12M)^d\sup_n\|K_n\|_\infty.
$$
If $|y|_\infty>2M$, then, for $z\in E-y$,
$$
|y|_\infty-M\leq|z|_\infty\leq|y|_\infty+M,
\qquad
\frac{|y|_\infty+M}{|y|_\infty-M}<3.
$$
Thus $E-y$ meets at most four of the sets $V_j$, so
$$
\int_E\sum_{R\in\cD}
\sup_n|a_{n,R}|\dsone_R(x-y)\,dx
\leq8C_V.
$$
Consequently, we have
$$
\sup_{y\in\bR^d}
\int_E\sum_{R\in\cD}
\sup_n|a_{n,R}|\dsone_R(x-y)\,dx
\leq2(12M)^d\sup_n\|K_n\|_\infty+8C_V<\infty.
$$
Since
$|Q|\|\bE_Qf\|_{\mathcal X}
 \leq\int_Q\|f(y)\|_{\mathcal X}\,dy$,
\eqref{eq:exact-cube-pairing} implies
\begin{align*}
    \int_E\sum_{Q\in\cQ}
      m_Q(x)\|\bE_Qf\|_{\mathcal X}\,dx
    &\leq
      \int_E\int_{\bR^d}\|f(y)\|_{\mathcal X}
      \sum_{R\in\cD}\sup_n|a_{n,R}|
      \sum_{Q\in\cQ}
      \dsone_{X_Q^R}(x)\dsone_Q(y)\,dy\,dx\\
    &=
      \int_{\bR^d}\|f(y)\|_{\mathcal X}
      \int_E\sum_{R\in\cD}
      \sup_n|a_{n,R}|\dsone_R(x-y)\,dx\,dy\\
    &\leq
      \left(
      2(12M)^d\sup_n\|K_n\|_\infty+8C_V
      \right)\|f\|_{L_1(\bR^d;\mathcal X)}
      <\infty.
\end{align*}
This proves \eqref{eq:input-cube-majorant}. By
\eqref{eq:exact-kernel-series} and
\eqref{eq:cube-average-representation},
\begin{align*}
    K_n*f
    =\sum_{R\in\cD}a_{n,R}(\dsone_R*f)=\sum_{R\in\cD}\sum_{Q\in\cQ}
      a_{n,R}|Q|\dsone_{X_Q^R}\bE_Qf=\sum_{Q\in\cQ}b_{n,Q}\bE_Qf
\end{align*}
almost everywhere and in $L_1(E;\mathcal X)$.
This proves \eqref{eq:input-cube-expansion}.

It remains to prove \eqref{eq:pairing-exterior-bound}.
Fix $P\in\cQ_j$, $\beta\geq3$, and $y\in P$ outside
the dyadic boundaries. If $Q\subseteq P$, $y\in Q$, and
$x\in X_Q^R\setminus\beta P$, then
$
x-y\in R, |x-y|_\infty>2^{-j},
$
and therefore $\rho_R>2^{-j}$.
For every $k\geq0$, there is exactly one cube
$Q\in\cQ_{j+k}$ with $y\in Q\subseteq P$.
Since only cubes with $\ell(Q)\leq\ell(R)$ contribute, we get
\begin{align*}
    \sum_{\substack{Q\subseteq P,\ y\in Q\\
                     \ell(Q)\leq\ell(R)}}\ell(Q)
    =\sum_{\substack{k\geq0\\2^{-j-k}\leq\ell(R)}}
      2^{-j-k}
    =2\min\{2^{-j},\ell(R)\}.
\end{align*}
By \eqref{eq:cube-pairing-volume} and the preceding identity, we have
\begin{align*}
    \sum_{\substack{Q\subseteq P\\y\in Q}}
      \frac1{|Q|}\int_{(\beta P)^c}m_Q(x)\,dx
    &=
      \sum_{R\in\cD}\sup_n|a_{n,R}|
      \sum_{\substack{Q\subseteq P\\y\in Q}}
      |X_Q^R\setminus\beta P|\\
    &\leq
      d\sum_{\substack{R\in\cD\\\rho_R>2^{-j}}}
      \ell(R)^{d-1}\sup_n|a_{n,R}|
      \sum_{\substack{Q\subseteq P,\ y\in Q\\
                       \ell(Q)\leq\ell(R)}}\ell(Q)\\
    &=
      2d\sum_{\substack{R\in\cD\\\rho_R>2^{-j}}}
      |R|\sup_n|a_{n,R}|
      \min\left\{1,\frac{2^{-j}}{\ell(R)}\right\}.
\end{align*}
Splitting the last sum according to
$\ell(R)\leq2^{-j}$ and $\ell(R)=2^{k-j}$, $k\geq1$,
and applying \eqref{eq:dyadic-displacement} and
\eqref{eq:pairing-scale-bound}, we obtain
\begin{align*}
    \sum_{\substack{R\in\cD\\\rho_R>2^{-j}}}
      |R|\sup_n|a_{n,R}|
      \min\left\{1,\frac{2^{-j}}{\ell(R)}\right\}
    &\leq
      \sum_{R\in\cD\cap\mathcal E_j}
      |R|\sup_n|a_{n,R}|
      +\sum_{k\geq1}2^{-k}
       \sum_{\substack{R\in\cD\\\ell(R)=2^{k-j}}}
       |R|\sup_n|a_{n,R}|\\
    &\leq
      C_{d,A}(C_V+C_D)
      \left(1+\sum_{k\geq1}2^{-k}\right).
\end{align*}
Combining these estimates proves
\eqref{eq:pairing-exterior-bound}.
\end{proof}

\section{Proofs of the main theorems}
\label{sec:maximal}

In this section, we prove Theorems~\ref{thm:joint}
and~\ref{thm:truncations}.
We first recall an equivalent form of the joint
$L_1$-H\"ormander condition and regularize the convolution
operators. These preliminary arguments are standard,
and we include the proofs for completeness.
We then apply the results of
Sections~\ref{sec:dyadic-weak}
and~\ref{sec:kernel-decomposition}
to prove the main theorems.

\subsection{The joint \texorpdfstring{$L_1$}{L1}-H\"ormander condition}
\label{sec:hormander}

Let $(K_\lambda)_{\lambda\in\Lambda}$ be a family of
measurable functions on $\bR^d\setminus\{0\}$.
We say that this family satisfies the \emph{joint
$L_1$-H\"ormander condition} if there exists a finite
constant $C_H$ such that
\begin{equation}\label{eq:hormander}\tag{H}
    \sup_{y\in\bR^d}
    \int_{|x|\geq2|y|}
    \sup_{\lambda\in\Lambda}
    |K_\lambda(x-y)-K_\lambda(x)|\,dx
    \leq C_H.
\end{equation}

For later use, we record an equivalent formulation in
terms of centred cubes with dyadic side lengths.
For $k\in\bZ$, let
\[
    Q_0^{(k)}=[-2^{-k-1},2^{-k-1})^d.
\]

\begin{proposition}\label{prop:hormander-implications}
Condition~\eqref{eq:hormander} holds if and only if there
exist constants $A\geq1$ and $C_D<\infty$ such that
\begin{equation}\label{eq:dyadic-hormander}
    \sup_{k\in\bZ}
    \sup_{y\in Q_0^{(k)}}
    \int_{(AQ_0^{(k)})^c}
    \sup_{\lambda\in\Lambda}
    |K_\lambda(x-y)-K_\lambda(x)|\,dx
    \leq C_D.
\end{equation}
\end{proposition}

\begin{proof}
Suppose first that \eqref{eq:hormander} holds.
If $y\in Q_0^{(k)}$ and $x\notin AQ_0^{(k)}$, where
$A>2\sqrt d$, then
\[
    |x|
    \geq\frac{A}{2}2^{-k}
    \geq\sqrt d\,2^{-k}
    \geq2|y|.
\]
Thus \eqref{eq:dyadic-hormander} follows with
$C_D\leq C_H$.

Conversely, suppose that \eqref{eq:dyadic-hormander} holds.
Fix $y\ne0$, set $m=\lceil4A\sqrt d\rceil$ and $y^\prime=y/m$,
and choose $k\in\bZ$ such that $ 2|y^\prime|_\infty<2^{-k}\leq4|y^\prime|_\infty.$
Then $y^\prime\in Q_0^{(k)}$, and every $z\in AQ_0^{(k)}$
satisfies $ |z|
    \leq\frac{|y|}{2}.$
If $|x|\geq2|y|$ and $0\leq j<m$, then
\[
    |x-jy^\prime|
    \geq|x|-j|y^\prime|
    >|y|,
\]
so $x-jy^\prime\notin AQ_0^{(k)}$.
By the triangle inequality and a change of variables, we have
\begin{align*}
    \int_{|x|\geq2|y|}
    \sup_{\lambda\in\Lambda}
    |K_\lambda(x-y)-K_\lambda(x)|\,dx
    &\leq
    \sum_{j=0}^{m-1}
    \int_{|x|\geq2|y|}
    \sup_{\lambda\in\Lambda}
    |K_\lambda(x-(j+1)y^\prime)-K_\lambda(x-jy^\prime)|\,dx\\
    &\leq
    m\int_{(AQ_0^{(k)})^c}
    \sup_{\lambda\in\Lambda}
    |K_\lambda(z-y^\prime)-K_\lambda(z)|\,dz
    \leq mC_D.
\end{align*}
The case $y=0$ is immediate.
Taking the supremum over $y$ proves
\eqref{eq:hormander}.
\end{proof}

\subsection{Regularization of convolution operators}
\label{subsec:regularization}

In this subsection, we regularize the convolution
operators in Theorem~\ref{thm:joint}.
The resulting kernels are bounded and Lipschitz,
so that we can apply the decomposition established
in Section~\ref{sec:kernel-decomposition}.

Let $(T_n)_{n\geq1}$ be convolution
operators with scalar distribution kernels $\sigma_n$,
as defined in Section~\ref{subsec:distribution-convolutions}.
Assume that the restriction of $\sigma_n$ to
$\bR^d\setminus\{0\}$ is represented by
$K_n\in L_{1,\mathrm{loc}}(\bR^d\setminus\{0\})$.
Thus, for $f\in\mathscr D$,
\[
    T_nf(x)=\int_{\bR^d}K_n(x-y)f(y)\,dy
    \quad\text{for almost every }
    x\notin\supp_{\bR^d}f.
\]
Suppose that $(K_n)_n$ satisfies
\eqref{eq:hormander} with constant $C_H$ and that,
for some $1<r\leq\infty$,
\[
    \|(T_nf)_n\|_{L_r(\cN;\ell_\infty)}
    \leq M_r\|f\|_r,
    \qquad f\in L_r(\cN).
\]

Fix a nonnegative function
$\psi\in C_c^\infty(B(0,1/8))$ with $\int\psi=1$.
For $\varepsilon>0$, put
\[
    \psi_\varepsilon(x)
    =\varepsilon^{-d}\psi(x/\varepsilon),
    \qquad
    \eta_\varepsilon
    =\psi_\varepsilon*\psi_\varepsilon,
\]
and define
\[
    K_{n,\varepsilon}(x)
    =\langle\sigma_n,\eta_\varepsilon(x-\cdot)\rangle,
    \qquad
    T_{n,\varepsilon}f=T_n(\eta_\varepsilon*f).
\]
Since $\supp\eta_\varepsilon\subset B(0,\varepsilon/4)$,
\begin{equation}\label{eq:regularized-off-origin}
    K_{n,\varepsilon}(x)
    =\int_{\bR^d}K_n(x-y)\eta_\varepsilon(y)\,dy,
    \qquad |x|>\varepsilon/4.
\end{equation}

The following estimates follow from standard 
arguments. We include a short proof to record their
uniformity in $\varepsilon$.

\begin{lemma}\label{lem:regularized-kernels}
For every $\varepsilon>0$, the family
$(K_{n,\varepsilon})_{n\geq1}$ satisfies
\begin{equation}\label{eq:regularized-smoothness}
    \sup_n\left(
        \varepsilon^d\|K_{n,\varepsilon}\|_\infty
        +\varepsilon^{d+1}
         \|\nabla K_{n,\varepsilon}\|_\infty
    \right)
    \leq C_{d,r}M_r.
\end{equation}
Moreover,
\begin{equation}\label{eq:regularized-hormander}
    \sup_{h\ne0}\int_{|x|>2|h|}
    \sup_n
    |K_{n,\varepsilon}(x-h)-K_{n,\varepsilon}(x)|\,dx
    \leq C_{d,r}(M_r+C_H),
\end{equation}
and
\begin{equation}\label{eq:regularized-shell}
    \sup_{j\in\bZ}
    \int_{V_j}\sup_n|K_{n,\varepsilon}(x)|\,dx
    \leq C_{d,r}(M_r+C_H).
\end{equation}
\end{lemma}

\begin{proof}
Choose $p\in\cP(\cM)$ with $0<\tau(p)<\infty$.
For $1<r<\infty$, and $u\in C_c^\infty(\bR^d)$
\begin{equation}\label{eq:scalar-maximal-bound}
    \begin{aligned}
        \left\|\sup_n|\sigma_n*u|\right\|_r
        =
        \tau(p)^{-1/r}
        \|(T_n(u\otimes p))_n
          \|_{L_r(\cN;\ell_\infty)}\leq M_r\|u\|_r,
    \end{aligned}
\end{equation}
see~\cite[Remark~2.3]{jungexu07ergodic}.
For $r=\infty$, the same estimate follows directly
by testing on $u\otimes p$.
Since
$K_{n,\varepsilon}
=(\sigma_n*\psi_\varepsilon)*\psi_\varepsilon$, by
H\"older's inequality, we get
\begin{align*}
    \|K_{n,\varepsilon}\|_\infty
    +\varepsilon\|\nabla K_{n,\varepsilon}\|_\infty
    \leq
    \|\sigma_n*\psi_\varepsilon\|_r
    \bigl(
        \|\psi_\varepsilon\|_{r'}
        +\varepsilon\|\nabla\psi_\varepsilon\|_{r'}
    \bigr)\leq C_{d,r}M_r\varepsilon^{-d},
\end{align*}
where $1/r+1/r'=1$.
This proves \eqref{eq:regularized-smoothness}.

For $0<a<b$ and $t>0$, note that
\[
    K_n(x)
    =(\sigma_n*\psi_{at})(x)
     +\int_{\bR^d}\psi_{at}(y)
       \bigl(K_n(x)-K_n(x-y)\bigr)\,dy,
    \qquad at<|x|<bt.
\]
Since $|x|>2|y|$ on this region whenever
$y\in\supp\psi_{at}$, equations
\eqref{eq:hormander} and \eqref{eq:scalar-maximal-bound}
imply
\begin{equation}\label{eq:common-annular-mass}
    \begin{aligned}
        \int_{at<|x|<bt}\sup_n|K_n(x)|\,dx
       \leq C_{d,r,a,b}M_r+C_H.
    \end{aligned}
\end{equation}
To prove \eqref{eq:regularized-hormander}, fix $z\ne0$.
By \eqref{eq:regularized-off-origin}, we have
\begin{align*}
    \int_{|x|>4\max\{|z|,\varepsilon\}}
    \sup_n
    |K_{n,\varepsilon}(x-z)-K_{n,\varepsilon}(x)|\,dx
    \leq C_H.
\end{align*}
Here we used $|x-y|>2|z|$ for
$y\in\supp\eta_\varepsilon$.
If $|z|\leq\varepsilon$, by
\eqref{eq:scalar-maximal-bound}, we get
\begin{align*}
\int_{2|z|<|x|\leq4\varepsilon}
    \sup_n
    |K_{n,\varepsilon}(x-z)-K_{n,\varepsilon}(x)|\,dx\leq
    2|B(0,5\varepsilon)|^{1/r'}
    M_r\|\eta_\varepsilon\|_r
    \leq C_{d,r}M_r.
\end{align*}
If $|z|>\varepsilon$, \eqref{eq:regularized-off-origin} yields
\begin{align*}
    \int_{2|z|<|x|\leq4|z|}
    \sup_n
    |K_{n,\varepsilon}(x-z)-K_{n,\varepsilon}(x)|\,dx\leq
    2\int_{3|z|/4<|x|<21|z|/4}
    \sup_n|K_n(x)|\,dx
\end{align*}
Combining the preceding estimates yields
\eqref{eq:regularized-hormander}.

Finally,
applying the argument for
\eqref{eq:common-annular-mass} to $K_{n,\varepsilon}$,
with \eqref{eq:regularized-hormander}, we have
\[
    \int_{at<|x|<bt}
    \sup_n|K_{n,\varepsilon}(x)|\,dx
    \leq C_{d,r,a,b}(M_r+C_H).
\]
Since
$V_j\subset\{x:2^{-j-1}<|x|<3\sqrt d\,2^{-j}\}$,
this proves \eqref{eq:regularized-shell}.
\end{proof}

We now verify that $K_{n,\varepsilon}$ is the convolution
kernel of $T_{n,\varepsilon}$ and that the original maximal
estimate remains valid.

\begin{proposition}\label{prop:regularization}
Let $1<r\leq\infty$ be the exponent in the assumed
maximal estimate.
For every $\varepsilon>0$ and
$f\in L_1(\cN)\cap L_r(\cN)$,
\begin{equation}\label{eq:regularized-convolution}
    T_{n,\varepsilon}f
    =K_{n,\varepsilon}*f
    \qquad\text{almost everywhere}.
\end{equation}
Moreover,
\begin{equation}\label{eq:regularized-maximal-bound}
    \|(T_{n,\varepsilon}f)_n\|_{L_r(\cN;\ell_\infty)}
    \leq M_r\|f\|_r,
    \qquad f\in L_r(\cN).
\end{equation}
\end{proposition}

\begin{proof}
For $r<\infty$, the identity extends to every
$f\in L_1(\cN)\cap L_r(\cN)$ by a density argument.
For $r=\infty$, \eqref{eq:scalar-maximal-bound} gives
$|\langle\sigma_n,u\rangle|
    \leq M_\infty\|u\|_\infty.$
By the Riesz representation theorem, $\sigma_n$ is a
finite complex measure.
Since it agrees with $K_n\,dx$ away from the origin, we have
\[
    \sigma_n=K_n\,dx+c_n\delta_0,
    \qquad
    \|K_n\|_1+|c_n|\leq M_\infty,
\]
where $\delta_0$ is the Dirac measure at the origin.
Moreover, $T_nf=K_n*f+c_nf$ for every $f\in\cN$
by weak$^*$ continuity.
Hence, we get
\[
    T_{n,\varepsilon}f
    =(K_n*\eta_\varepsilon+c_n\eta_\varepsilon)*f
    =K_{n,\varepsilon}*f,
\]
which completes the proof.
\end{proof}

The preceding estimates allow us to apply
Propositions~\ref{prop:dyadic-majorants}
and~\ref{prop:kernel-pairing} to the regularized kernels,
with coefficient bounds independent of $\varepsilon$.
We can therefore use Theorem~\ref{prop:exterior-endpoint}
to prove Theorem~\ref{thm:joint}.

\subsection{Proof of Theorem~\ref{thm:joint}}

Theorem~\ref{prop:exterior-endpoint} gives a weak type
$(1,1)$ estimate for $(T_{n,\varepsilon})_n$, with a
constant independent of $\varepsilon$.
Letting $\varepsilon\to0$, we obtain the same type of
estimate for $(T_n)_n$ on $\mathscr D$.
We then extend each $T_n$ to $L_1(\cN)$ and use approximation to obtain the weak
type $(1,1)$ estimate on all of $L_1(\cN)$.

\begin{proof}[Proof of Theorem~\ref{thm:joint}]
If $M_r=0$, the conclusion is immediate.
Assume that $M_r>0$.
For $\varepsilon>0$, let
\[
    T_{n,\varepsilon}f=T_n(\eta_\varepsilon*f).
\]
The estimates in Subsection~\ref{subsec:regularization},
together with Propositions~\ref{prop:hormander-implications},
\ref{prop:dyadic-majorants} and
\ref{prop:kernel-pairing}, give a dyadic representation
of $(T_{n,\varepsilon})_n$ for which
\[
    \Gamma\leq C_{d,r}(M_r+C_H),
    \qquad \beta=3.
\]
By Proposition~\ref{prop:regularization}, this representation
agrees with $T_{n,\varepsilon}$ on $\mathscr D_0$, and
\[
    \|(T_{n,\varepsilon}f)_n\|_{L_r(\cN;\ell_\infty)}
    \leq M_r\|f\|_r.
\]
By Theorem~\ref{prop:exterior-endpoint}, we can find
a constant $A>0$, independent of $\varepsilon$, such that
$A\leq C_{d,r}(M_r+C_H)$ and, for every
$f\in\mathscr D$ and $\lambda>0$, there is a projection
$e\in\cP(\cN)$ satisfying
\begin{equation}\label{eq:regularized-weak}
    \sup_n\|eT_{n,\varepsilon}(f)e\|_\infty\leq\lambda,
    \qquad
    \lambda\varphi(1-e)\leq A\|f\|_1.
\end{equation}

Fix $f\in\mathscr D$ and $\lambda>0$.
If $r<\infty$, then
\[
    \|(T_n(f-\eta_\varepsilon*f))_n
      \|_{L_r(\cN;\ell_\infty)}
    \leq M_r\|f-\eta_\varepsilon*f\|_r
    \longrightarrow0
    \qquad(\varepsilon\to0).
\]
By Lemma~\ref{lemma:factor-cut}, for sufficiently small
$\varepsilon$ there is a projection $e_1$ such that
\[
    \sup_n
    \|e_1T_n(f-\eta_\varepsilon*f)e_1\|_\infty
    \leq\frac{\lambda}{2},
    \qquad
    \varphi(1-e_1)\leq\frac{A}{\lambda}\|f\|_1.
\]
By \eqref{eq:regularized-weak}, there is also a projection
$e_2$ such that
\[
    \sup_n\|e_2T_{n,\varepsilon}(f)e_2\|_\infty
    \leq\frac{\lambda}{2},
    \qquad
    \varphi(1-e_2)\leq\frac{2A}{\lambda}\|f\|_1.
\]
Set $e=e_1\wedge e_2$.
Since
\[
    T_nf
    =T_{n,\varepsilon}f
     +T_n(f-\eta_\varepsilon*f),
\]
we obtain
\begin{equation}\label{eq:dense-weak-estimate}
    \sup_n\|eT_n(f)e\|_\infty\leq\lambda,
    \qquad
    \lambda\varphi(1-e)\leq3A\|f\|_1.
\end{equation}

For $r=\infty$, the same conclusion follows from
\[
    \sup_n\|T_n(f-\eta_\varepsilon*f)\|_\infty
    \leq M_\infty\|f-\eta_\varepsilon*f\|_\infty
    \longrightarrow0,
\]
since $f$ is a finite sum of smooth compactly supported
scalar functions tensored with bounded operators.
A standard density argument gives the extensions of
$T_n$ to $L_1(\cN)$ and the weak type $(1,1)$ estimate
on $L_1(\cN)$ with constant at most $CA$.
Since $A\leq C_{d,r}(M_r+C_H)$, the proof is complete.
\end{proof}

\subsection{Proof of Theorem~\ref{thm:truncations}}
\label{pf of B}

We apply Theorem~\ref{thm:joint} to the dyadic truncations
and then pass to arbitrary radii using the size condition
and the noncommutative Hardy--Littlewood maximal inequality.
This reduction follows the argument in
\cite[Section~3.2]{honglaixu23sigular}, with $T_{2^k}$
in place of the smooth truncations used there.

\begin{lemma}\label{lemma:dyadic-truncations}
Suppose that $|K(x)|\leq A|x|^{-d}$ for $x\ne0$ and
that $K$ satisfies the $L_1$-H\"ormander condition
with constant $C_H$.
Then the kernels
$K_k=K\dsone_{\{|x|>2^k\}}$, $k\in\bZ$, satisfy
$$
\sup_{h\ne0}\int_{|x|>2|h|}
\sup_{k\in\bZ}|K_k(x-h)-K_k(x)|\,dx
\leq C_H+C_dA.
$$
\end{lemma}

\begin{proof}
For every $k\in\bZ$, one has
$$
\begin{aligned}
    |K_k(x-h)-K_k(x)|
    \leq |K(x-h)-K(x)|+
    A|x|^{-d}
    \left|
        \dsone_{\{|x-h|>2^k\}}
        -\dsone_{\{|x|>2^k\}}
    \right|.
\end{aligned}
$$
Then, the H\"ormander condition therefore yields
$$
\begin{aligned}
    &\int_{|x|>2|h|}
    \sup_k|K_k(x-h)-K_k(x)|\,dx\leq
    C_H+C_dA
    \left(
        1+\sum_{\substack{k\in\bZ\\2^k>2|h|}}
        \frac{|h|}{2^k}
    \right)
    \leq C_H+C_dA.
\end{aligned}
$$
\end{proof}

For $f\in L_1(\cN)_+$, define
$$
M_kf(x)
=2^{-d(k+1)}
 \int_{x+[-2^k,2^k]^d}f(y)\,dy,
\qquad k\in\bZ.
$$
By the noncommutative Hardy--Littlewood maximal inequality
\cite[Theorem~4.1]{hongwangliao2017noncommutative},
applied to $\bR^d$ with the metric $|x-y|_\infty$
and Lebesgue measure, for every $f\in L_1(\cN)_+$
and $t>0$ there is a projection $e\in\cP(\cN)$ such that
\begin{equation}\label{eq:centered-cube-weak}
    \sup_{k\in\bZ}\|eM_k(f)e\|_\infty\leq t,
    \qquad
    \varphi(1-e)\leq\frac{C_d}{t}\|f\|_1.
\end{equation}

\begin{proof}[Proof of Theorem~\ref{thm:truncations}]
The case $A=0$ is immediate, so assume that $A>0$.
Restricting the assumed maximal strong estimate to
$\varepsilon=2^k$ and applying
Lemma~\ref{lemma:dyadic-truncations} and
Theorem~\ref{thm:joint}, we obtain a weak type $(1,1)$
estimate for $(T_{2^k})_{k\in\bZ}$ with constant at most
$C_{d,p_0}(M_{p_0}+C_H+A)$.
Since $K_k\in L_\infty(\bR^d)$, approximation by
$\mathscr D$ in $L_1$ identifies these extensions
with the truncated integrals.

Fix $f\in L_1(\cN)$ and $2^k\leq\varepsilon<2^{k+1}$.
The kernel of $T_\varepsilon-T_{2^k}$ satisfies
$$
\left|
K(x)\left(
    \dsone_{\{|x|>\varepsilon\}}
    -\dsone_{\{|x|>2^k\}}
\right)
\right|
\leq
A2^{-kd}\dsone_{[-2^{k+1},2^{k+1}]^d}(x).
$$
Write
$r_{\varepsilon,k}
=K(\dsone_{\{|\cdot|>\varepsilon\}}
   -\dsone_{\{|\cdot|>2^k\}})$.
For $|c|\leq1$, we have
$$
\begin{pmatrix}
    |f(y)^*|&cf(y)\\
    \overline c f(y)^*&|f(y)|
\end{pmatrix}\geq0.
$$
Since
$|r_{\varepsilon,k}(y)|
\leq A2^{-kd}\dsone_{[-2^{k+1},2^{k+1}]^d}(y)$, using
$$
A2^{-kd}\int_{[-2^{k+1},2^{k+1}]^d}|f(x-y)|\,dy
=4^dA\,M_{k+1}(|f|)(x),
$$
and the corresponding identity for $|f^*|$, we have
\begin{equation}\label{eq:truncation-positive-block}
    \begin{pmatrix}
        4^dA M_{k+1}(|f^*|)
        &(T_\varepsilon-T_{2^k})f\\
        ((T_\varepsilon-T_{2^k})f)^*
        &4^dA M_{k+1}(|f|)
    \end{pmatrix}\geq0.
\end{equation}

For $t>0$, apply \eqref{eq:centered-cube-weak} to
$|f^*|$ and $|f|$ at height $t/(4^dA)$.
We obtain projections $e_1,e_2$ such that
$$
\sup_k\|e_1M_k(|f^*|)e_1\|_\infty
\leq\frac{t}{4^dA},
\qquad
\sup_k\|e_2M_k(|f|)e_2\|_\infty
\leq\frac{t}{4^dA},
$$
and
$$
\varphi(1-e_1)+\varphi(1-e_2)
\leq\frac{C_dA}{t}\|f\|_1.
$$
Set $e_3=e_1\wedge e_2$. By
\eqref{eq:truncation-positive-block},
$$
\begin{pmatrix}
    4^dA\,e_3M_{k+1}(|f^*|)e_3
    &e_3(T_\varepsilon-T_{2^k})(f)e_3\\
    \bigl(e_3(T_\varepsilon-T_{2^k})(f)e_3\bigr)^*
    &4^dA\,e_3M_{k+1}(|f|)e_3
\end{pmatrix}\geq0.
$$
Since $e_3\leq e_1,e_2$, both diagonal entries are
bounded above by $te_3$.
By Lemma~\ref{lemma:positive-block}, we get
\begin{equation}\label{eq:truncation-difference-weak}
    \sup_{k\in\bZ}
    \sup_{2^k\leq\varepsilon<2^{k+1}}
    \|e_3(T_\varepsilon-T_{2^k})(f)e_3\|_\infty
    \leq t,
    \qquad
    \varphi(1-e_3)\leq\frac{C_dA}{t}\|f\|_1.
\end{equation}

Finally, fix $\lambda>0$.
By the dyadic weak type estimate, we can find a projection $e_0$
such that
$$
\sup_k\|e_0T_{2^k}(f)e_0\|_\infty
\leq\frac{\lambda}{2},
\qquad
\varphi(1-e_0)
\leq\frac{C_{d,p_0}}{\lambda}
(M_{p_0}+C_H+A)\|f\|_1.
$$
Choose $e_3$ in \eqref{eq:truncation-difference-weak}
with $t=\lambda/2$ and set $e=e_0\wedge e_3$.
For $2^k\leq\varepsilon<2^{k+1}$,
$$
\begin{aligned}
    \|eT_\varepsilon(f)e\|_\infty
    \leq
    \|eT_{2^k}(f)e\|_\infty
    +\|e(T_\varepsilon-T_{2^k})(f)e\|_\infty\leq\lambda.
\end{aligned}
$$
Moreover,
$$
\begin{aligned}
    \lambda\varphi(1-e)
    \leq
    \lambda\varphi(1-e_0)+\lambda\varphi(1-e_3)\leq
    C_{d,p_0}(M_{p_0}+C_H+A)\|f\|_1.
\end{aligned}
$$
This proves the theorem.
\end{proof}

\section{Applications I: Maximal inequalities and approximation of the identity}
\label{sec:approximation}

In this section, we use Theorem~\ref{thm:joint} to obtain
maximal inequalities for integrable approximation kernels
and multiparameter Poisson averages on Lie groups.
The examples include kernels with singularities away from
the origin that do not satisfy the $L_2$ integral regularity
condition.
For some of these kernels, Hardy--Littlewood domination
also fails. In the Poisson application, a different issue
arises: direct Hardy--Littlewood domination gives a linear
loss for displaced interval averages, while the summation
in the proof requires a logarithmic bound. Our theorem
provides this logarithmic bound through the joint
$L_1$-H\"ormander condition.

We use the classical scalar kernel estimates of
Z\'o~\cite{zo76appro}, Stein~\cite{stein93bookharmonic},
and Sj\"ogren~\cite{sjogren86admissible} without repeating
their proofs. Applying our theorem to the kernels,
gives the noncommutative estimates used below.

\subsection{Integrable kernels and dyadic approximation}

We first establish maximal inequalities for integrable
kernels. Applied to dilations of a fixed kernel, these
estimates will yield bilateral almost uniform convergence
for dyadic approximations of the identity.

\begin{corollary}\label{thm:main-hormander}
Let $(K_\lambda)_{\lambda\in\Lambda}\subset L_1(\bR^d)$
satisfy
$
\sup_{\lambda\in\Lambda}\|K_\lambda\|_1<\infty
$
and the joint H\"ormander condition \eqref{eq:hormander}.
Then the family $(T_\lambda)_{\lambda\in\Lambda}$,
defined by $T_\lambda f=K_\lambda*f$, is of weak
type $(1,1)$ and of strong type $(p,p)$
for $1<p\leq\infty$.
\end{corollary}

\begin{proof}
Without loss of generality, we may assume that
$\Lambda=\bN$ and $f\geq0$.
Since
$$
\sup_n\|T_nf\|_\infty
\leq\left(\sup_n\|K_n\|_1\right)\|f\|_\infty,
$$
Theorem~\ref{thm:joint}, with $r=\infty$, gives
the weak type $(1,1)$ estimate.

Since
$$
\bigl||K_n(x-y)|-|K_n(x)|\bigr|
\leq|K_n(x-y)-K_n(x)|.
$$
Applying Theorem~\ref{thm:joint} and
Lemma~\ref{lemma:positive-interpolation} to
$f\mapsto |K_n|*f$, for $1<p<\infty$, there exists $U\in L_p(\cN)_+$ such that
$$
|K_n|*f\leq U,
\qquad
\|U\|_p\leq C_p\|f\|_p.
$$
As
$$
-U\leq\operatorname{Re}(T_nf)\leq U,
\qquad
-U\leq\operatorname{Im}(T_nf)\leq U,
$$
By Remark~\ref{rk:MaxFunct}, we have
$$
\|(T_nf)_n\|_{L_p(\cN;\ell_\infty)}
\leq2\|U\|_p
\leq2C_p\|f\|_p.
$$
This proves the strong type estimate.
\end{proof}

We give a sufficient condition for the dyadic dilations
of an integrable kernel to satisfy
\eqref{eq:hormander}.
\begin{proposition}\label{prop:dini-hormander}
Let $K\in L_1(\bR^d)$.
Suppose that there exist $A\geq1$ and a nonnegative,
nondecreasing function $\eta$ on $[0,1]$, with
$\eta(0)=0$ and
\(
    \int_0^1\eta(s)\,\frac{ds}{s}<\infty,
\)
such that
\begin{equation}\label{eq:L1 Dini}
    \sup_{y\in Q_0^{(k)}}
    \int_{(AQ_0^{(k)})^c}
    |K(x-y)-K(x)|\,dx
    \leq\eta(2^{-k})
    \qquad(k\geq0),
\end{equation}
and
\begin{equation}\label{eq:size condition}
    \sup_{y\in Q_0^{(k)}}
    \int_{(AQ_0^{(k)})^c}|K(x-y)|\,dx
    \leq\eta(2^k)
    \qquad(k\leq0).
\end{equation}
Then the kernels
\[
    K_j(x)=2^{-jd}K(2^{-j}x),
    \qquad j\in\bZ,
\]
satisfy \eqref{eq:hormander}.
\end{proposition}

\begin{proof}
Fix $k\in\bZ$ and $y\in Q_0^{(k)}$.
Changing variables $x=2^jz$ gives
\[
    \int_{(AQ_0^{(k)})^c}
    |K_j(x-y)-K_j(x)|\,dx
    =
    \int_{(AQ_0^{(k+j)})^c}
    |K(z-2^{-j}y)-K(z)|\,dz.
\]
By \eqref{eq:L1 Dini} and
\eqref{eq:size condition}, we deduce that
\begin{align*}
    \int_{(AQ_0^{(k)})^c}
    \sup_{j\in\bZ}|K_j(x-y)-K_j(x)|\,dx
    &\leq
    \sum_{j\geq1-k}\eta(2^{-k-j})
    +2\sum_{j\leq-k}\eta(2^{k+j})\\
    &\quad\lesssim
    \eta(1)+\int_0^1\eta(s)\,\frac{ds}{s}.
\end{align*}
Here we used monotonicity of $\eta$ and the fact
\[
    \eta(2^{-r})\log2
    \leq\int_{2^{-r}}^{2^{-r+1}}\eta(s)\,\frac{ds}{s},
    \qquad r\geq1.
\]
Taking the supremum over $k$ and $y$, and applying
Proposition~\ref{prop:hormander-implications},
proves \eqref{eq:hormander}.
\end{proof}

For $K\in L_1(\bR^d)$ and $t>0$, put
$$
K_t(x)=t^{-d}K(t^{-1}x),
\qquad
T_tf=K_t*f.
$$
The conditions in Proposition~\ref{prop:dini-hormander}
allow us to apply Corollary~\ref{thm:main-hormander} to
dyadic dilations.

\begin{corollary}
\label{cor:dyadic for dini}
\label{cor:continuous-dini}
Suppose that $K$ satisfies \eqref{eq:L1 Dini} and
\eqref{eq:size condition}.
Then $(T_{2^j})_{j\in\bZ}$ is of weak type $(1,1)$
and of strong type $(p,p)$ for $1<p\leq\infty$.

If, in addition, $K$ is real-valued and
$r\mapsto K(rx)$ is nonincreasing on $(0,\infty)$
for almost every $x$, the same conclusions hold
for $(T_t)_{t>0}$.
\end{corollary}

\begin{proof}
Since $\|K_{2^j}\|_1=\|K\|_1$, the dyadic assertion
follows from Proposition~\ref{prop:dini-hormander}
and Corollary~\ref{thm:main-hormander}.
Under the additional assumption, integrability and
monotonicity along rays imply $K\geq0$ almost everywhere.
Without loss of generality, we may assume that $f\geq0$.
For $2^j\leq t<2^{j+1}$,
$$
0\leq K_t\leq2^dK_{2^{j+1}},
\qquad
0\leq T_tf\leq2^dT_{2^{j+1}}f.
$$
The weak and strong type estimates therefore follow
from the dyadic case.
\end{proof}

\begin{remark}\label{remark:singular-kernels}
For comparison, the $L_2$ integral regularity condition
in \cite[Theorem~A]{cadilhaccondeparcet22spectral}
takes the following form for convolution kernels:
\begin{equation}\label{eq:previous-L2-condition}
    \sup_{Q\in\cQ}
    \sum_{m\geq1}\sup_{y\in Q}
    \left(
        (2^m\ell(Q))^d
        \int_{\substack{2^m\ell(Q)\leq|x-c_Q|\\
                        |x-c_Q|\leq2^{m+1}\ell(Q)}}
        |K(x-y)-K(x-c_Q)|^2\,dx
    \right)^{1/2}<\infty.
\end{equation}
The following kernels satisfy \eqref{eq:L1 Dini}
and \eqref{eq:size condition}, but fail
\eqref{eq:previous-L2-condition}.

\smallskip
\noindent\textup{(i)}
For $0<\varepsilon\leq1/2$,
$$
K(x)=(1-|x|^2)^{\varepsilon-1}
     \dsone_{\{|x|<1\}}(x).
$$
In dimension one, this gives
$K(x)=(1-x^2)^{\varepsilon-1}\dsone_{(-1,1)}(x)$.

\smallskip
\noindent\textup{(ii)}
For $a=(1,0,\ldots,0)$ and $d/2\leq\alpha<d$,
$$
K(x)=|x-a|^{-\alpha}
     \dsone_{\{0<|x-a|<1/4\}}(x).
$$

\smallskip
\noindent\textup{(iii)}
For $1/2\leq\alpha<1$,
$$
K(x)=|x_1-1|^{-\alpha}
     \dsone_{\{0<|x_1-1|<1/4\}}
     \prod_{i=2}^d\dsone_{\{|x_i|<1/4\}}.
$$

\smallskip
\noindent\textup{(iv)}
In dimension one, for $\beta>2$,
$$
K(x)=
\frac{\dsone_{\{0<|x-1|<1/4\}}}
{|x-1|\bigl(\log(e/|x-1|)\bigr)^\beta},
$$
with $K(1)=0$.

Each kernel has a singularity away from the origin
that is not locally square integrable. For sufficiently
small $h=(t,0,\ldots,0)$ with $t>0$,
$K(\,\cdot-h)-K$ is not square integrable on a
bounded region separated from the origin.
Thus an annular integral in
\eqref{eq:previous-L2-condition} is infinite.
\end{remark}

We now combine the dyadic maximal estimate with
the noncommutative Banach principle to obtain
bilateral almost uniform convergence.

\begin{corollary}\label{cor:approximate-identity}
Let $K\in L_1(\bR^d)$ satisfy \eqref{eq:L1 Dini}
and \eqref{eq:size condition}, and put
$c_K=\int_{\bR^d}K$.
For every $f\in L_1(\cN)$,
\begin{equation}\label{eq:dyadic-approximation}
    K_{2^{-n}}*f\longrightarrow c_Kf
    \quad\text{in }L_1(\cN)\text{ and b.a.u.}
    \qquad(n\to\infty).
\end{equation}
If $K$ has compact support, then the convergence
also holds locally b.a.u. for every
$f\in L_{1,\mathrm{loc}}(\bR^d;L_1(\cM))$.
\end{corollary}

\begin{proof}
Firstly, we have
$$
\|K_{2^{-n}}*f-c_Kf\|_1
\leq
\int_{\bR^d}|K(y)|
\|f(\,\cdot-2^{-n}y)-f\|_1\,dy
\longrightarrow0.
$$
And for $g=\sum_{i=1}^m u_i\otimes a_i\in\mathscr D$,
we also have
$$
\begin{aligned}
\|K_{2^{-n}}*g-c_Kg\|_\infty
\leq
\sum_{i=1}^m\|a_i\|_\infty
\int_{\bR^d}|K(y)|
\|u_i(\,\cdot-2^{-n}y)-u_i\|_\infty\,dy\longrightarrow0.
\end{aligned}
$$
By Corollary~\ref{cor:dyadic for dini}, the operators
$f\mapsto K_{2^{-n}}*f$, $n\geq0$, satisfy a weak
type $(1,1)$ estimate and are therefore bilaterally
uniformly equicontinuous at zero.
By the noncommutative Banach principle
\cite[Theorem~3.4]{chilinlitvinov06equicontinuity},
the set
$$
\left\{
    f\in L_1(\cN):
    (K_{2^{-n}}*f)_{n\geq0}
    \text{ converges b.a.u.}
\right\}
$$
is closed in $L_1(\cN)$.
It contains the dense subspace $\mathscr D$ and
therefore equals $L_1(\cN)$.
The convergence in $L_1(\cN)$ identifies the
limit as $c_Kf$.

Suppose now that $\supp K\subseteq[-R_0,R_0]^d$.
For $R>0$ and $n\geq0$,
$$
\dsone_{[-R,R]^d}(K_{2^{-n}}*f)
=
\dsone_{[-R,R]^d}
\left(
K_{2^{-n}}*
\bigl(\dsone_{[-R-R_0,R+R_0]^d}f\bigr)
\right).
$$
The truncated function belongs to $L_1(\cN)$.
Applying the global assertion and restricting the
resulting projections to $[-R,R]^d$ proves local
b.a.u. convergence.
\end{proof}
In particular, the compactly supported kernels in
Remark~\ref{remark:singular-kernels} satisfy both
the global and local convergence conclusions.

\subsection{Monotone $L\log L$ majorants}
\label{sec:logl}
Throughout this subsection, we work in dimension one
and write $\cN=L_\infty(\bR)\overline\otimes\cM$.
The following result extends Z\'o's monotone-kernel theorem
\cite[Theorem~3, pp.~116--117]{zo76appro}.

Let $L:(0,1)\to[0,\infty)$ be nondecreasing and satisfy
$$
\int_0^1 L(t)\log(1+L(t))\,dt<\infty.
$$
Extend $L$ by zero outside $(0,1)$, and let
$K:\bR\to\bC$ be measurable with $|K|\leq L$
almost everywhere.
For a sequence of positive numbers
$(\lambda_n)_{n\geq1}$, define
$$
K_n(t)=\lambda_nK(\lambda_nt),
\qquad
T_nf=K_n*f,
\qquad
c_K=\int_\bR K(t)\,dt.
$$

\begin{theorem}\label{thm:logl}
Suppose that $\lambda_n\leq q\lambda_{n+1}$
for some $0<q<1$ and every $n\geq1$.
Then $(T_n)_{n\geq1}$ is of weak type $(1,1)$
and of strong type $(p,p)$ for $1<p\leq\infty$,
with constants depending only on $L,p,q$.
Moreover, for $f\in L_1(\cN)$, 
$$
T_nf\longrightarrow c_Kf
\quad\text{in }L_1(\cN)\text{ and b.a.u.}.
$$
For $f\in L_{1,\mathrm{loc}}(\bR;L_1(\cM))$,
the convergence holds locally b.a.u.
\end{theorem}

\begin{proof}
The scalar estimate in the proof of
\cite[Theorem~3, pp.~116--117]{zo76appro} gives
$$
\sup_{y\ne0}\int_{|x|\geq2|y|}
\sup_{n\geq1}|g_n(x-y)-g_n(x)|\,dx
\leq C_{L,q},
\qquad
g_n(t)=\lambda_nL(\lambda_nt).
$$
Since $\|g_n\|_1=\|L\|_1$,
Corollary~\ref{thm:main-hormander} gives the weak
and strong maximal estimates for $S_nf=g_n*f$.
Since $|K_n|\leq g_n$, the weak and strong maximal
estimates for $(T_n)_n$ follow from those for $(S_n)_n$
by the same domination argument as in the proof of
Corollary~\ref{thm:main-hormander}.

On the other hand, since $\lambda_n\to\infty$, we have
$$
\|T_nf-c_Kf\|_1
\leq
\int_0^1|K(u)|
\|f(\,\cdot-u/\lambda_n)-f\|_1\,du
\longrightarrow0.
$$
For $f\in\mathscr D$, the analogous estimate in
operator norm also tends to zero.
The weak maximal estimate and the Banach-principle
argument in Corollary~\ref{cor:approximate-identity}
therefore give b.a.u. convergence for every
$f\in L_1(\cN)$.

Finally, $\supp K_n\subseteq[0,\lambda_1^{-1}]$
for all $n$. On a bounded interval $I$, one may
replace a locally integrable $f$ by its restriction
to $[\inf I-\lambda_1^{-1},\sup I]$.
The localization argument in
Corollary~\ref{cor:approximate-identity}
then gives local b.a.u. convergence.
\end{proof}

\begin{remark}\label{remark:endpoint-kernel}
The one-dimensional endpoint kernel from
Remark~\ref{remark:singular-kernels},
$$
K(x)=(1-x^2)^{\varepsilon-1}\dsone_{(-1,1)}(x),
\qquad 0<\varepsilon<1,
$$
also gives an application of Theorem~\ref{thm:logl};
see \cite[Chapter~II, \S5.16]{stein93bookharmonic}
for the classical example.
Indeed, the restriction of $K$ to $(0,1)$,
$$
L(t)=(1-t^2)^{\varepsilon-1},
\qquad 0<t<1,
$$
is nondecreasing and belongs to $L\log L(0,1)$.
Extend $L$ by zero outside $(0,1)$.
Since $K(x)=L(x)+L(-x)$, Theorem~\ref{thm:logl}
and the change of variables $x\mapsto -x$ give
maximal inequalities and global and local b.a.u.
convergence for the convolution operators with kernels
$\lambda_nK(\lambda_n\,\cdot)$, along every sequence
$(\lambda_n)$ satisfying the assumptions of that theorem.
\end{remark}

\subsection{Nilpotent groups and Poisson kernels}
\label{subsec:nilpotent-poisson}

We first prove a logarithmic estimate for displaced
interval averages. We then apply this estimate along
one-parameter subgroups to obtain maximal inequalities
on nilpotent groups and for Poisson integrals on
symmetric spaces.

For the interval estimate, let
$\cN=L_\infty(\bR)\overline\otimes\cM$.
For $b\in\bR$ and $\ell>0$, define
$$
B_k^{b,\ell}f(t)
=\frac1\ell
 \int_{b-\ell/2}^{b+\ell/2}
 f(t-e^{-k}s)\,ds,
\qquad k\in\bZ.
$$
The following estimate is the noncommutative version
of the classical displaced-interval estimate; see
\cite[Chapter~II, \S5.10, p.~78]{stein93bookharmonic}
and
\cite[Lemma~4, pp.~340--341]
{carlssonsjogrenstromberg85multiparameter}.

\begin{lemma}\label{lemma:displaced-exponential}
For every $1<p<\infty$ and $f\in L_p(\cN)_+$,
there exists $g\in L_p(\cN)_+$ such that
\begin{equation}\label{eq:displaced-exponential-majorant}
    B_k^{b,\ell}f\leq g
    \quad(k\in\bZ),
    \qquad
    \|g\|_p
    \leq
    C_p\bigl(1+\log(2+|b|/\ell)\bigr)^{1/p}\|f\|_p.
\end{equation}
\end{lemma}

\begin{proof}
The operators $B_k^{b,\ell}$ are convolutions with
the normalized interval kernels
$$
K_k(x)
=\frac{e^k}{\ell}
 \dsone_{[e^{-k}(b-\ell/2),\,e^{-k}(b+\ell/2)]}(x).
$$
The interval-kernel argument in
\cite[Chapter~II, \S5.10, p.~78]{stein93bookharmonic}
applies also to the scales $e^{-k}$ and gives
$$
\sup_{y\ne0}\int_{|x|\geq2|y|}
\sup_{k\in\bZ}|K_k(x-y)-K_k(x)|\,dx
\leq C\bigl(1+\log(2+|b|/\ell)\bigr).
$$
Since $\|K_k\|_1=1$, Theorem~\ref{thm:joint},
with $r=\infty$, gives the weak type $(1,1)$ estimate.
Lemma~\ref{lemma:positive-interpolation} then gives
\eqref{eq:displaced-exponential-majorant}.
\end{proof}

We now follow the cube decomposition and summation in
\cite[Proposition~5.1, pp.~325--326]{sjogren86admissible}.
Let $N$ be a connected, simply connected nilpotent
Lie group. Choose a basis $X_1,\ldots,X_d$ of its
Lie algebra such that
$\operatorname{span}\{X_1,\ldots,X_i\}$ is an ideal
for every $i$.
Following \cite[p.~323]{sjogren86admissible}, use
the canonical coordinates
$$
n=\exp(x_1X_1)\cdots\exp(x_dX_d),
\qquad
dn=dx_1\cdots dx_d,
$$
and define
$$
\delta_h(x_1,\ldots,x_d)
=(e^{h_1}x_1,\ldots,e^{h_d}x_d),
\qquad h\in\bZ^d.
$$
These coordinate dilations need not be group
automorphisms.
Put
$\cA=L_\infty(N)\overline\otimes\cM$,
with trace $\int_Ndn\otimes\tau$.

Let $P:N\to(0,\infty)$ be such that $1/P$ is a strictly
positive polynomial in the canonical coordinates and
$$
P(n)\leq C\min\{1,|x_1|^{-\eta},\ldots,|x_d|^{-\eta}\}
$$
for some $C,\eta>0$. Assume also that there exists
$\beta_0>0$ such that $P^\beta\in L_1(N)$ whenever
$\beta>\beta_0$.
Fix $\beta>\beta_0$ and define
$$
T_hF(g)
=\int_N
 F\bigl(g(\delta_{-h}n)^{-1}\bigr)P(n)^\beta\,dn,
\qquad h\in\bZ^d.
$$
Since $P^\beta\in L_1(N)$ and right translations
preserve Haar measure, these integrals define bounded
operators on $L_p(\cA)$ for $1\leq p<\infty$.
The following proposition gives a maximal inequality
for the family $(T_h)_{h\in\bZ^d}$.

\begin{proposition}\label{prop:nilpotent-poisson}
For every $1<p<\infty$ and $F\in L_p(\cA)$,
$$
\|(T_hF)_{h\in\bZ^d}\|_{L_p(\cA;\ell_\infty(\bZ^d))}
\leq C_{p,N,P,\beta}\|F\|_p.
$$
Moreover, if $F\geq0$, there exists $U\in L_p(\cA)_+$
such that
$$
T_hF\leq U\quad(h\in\bZ^d),
\qquad
\|U\|_p\leq C_{p,N,P,\beta}\|F\|_p.
$$
\end{proposition}

\begin{proof}
We first construct a common positive majorant for
$(T_hF)_{h\in\bZ^d}$ when $F\geq0$.
The estimate for general $F$ will then follow from
Lemma~\ref{lemma:positive-block}.
For each $i$, define
$$
\alpha_t^{(i)}F(g)=F(g\exp(-tX_i)),
\qquad t\in\bR.
$$
Since $N$ is unimodular, these right translations
preserve Haar measure. They therefore form a
weak$^*$ continuous trace-preserving action of $\bR$
on $\cA$.
Applying Lemma~\ref{lemma:positive-transference}
to this action and using
Lemma~\ref{lemma:displaced-exponential}, we deduce that for every $b\in\bR$, $\ell>0$,
$1<p<\infty$ and $F\in L_p(\cA)_+$, there exists
$G\in L_p(\cA)_+$ such that
\begin{equation}\label{eq:transferred-interval-majorant}
    \frac1\ell
    \int_{b-\ell/2}^{b+\ell/2}
    F\bigl(g\exp(-e^{-k}sX_i)\bigr)\,ds
    \leq G(g)
    \qquad(k\in\bZ),
\end{equation}
and
\begin{equation}\label{eq:transferred-interval-norm}
    \|G\|_p
    \leq
    C_p\bigl(1+\log(2+|b|/\ell)\bigr)^{1/p}\|F\|_p.
\end{equation}

Without loss of generality, we may assume that
$P\leq1/2$.
Set
$$
E_m=\{n\in N:P(n)>2^{-m}\},
\qquad m\geq1.
$$
The scalar construction in
\cite[p.~325]{sjogren86admissible}
gives pairwise disjoint coordinate cubes
$Q_j^m$, $1\leq j\leq J_m$, of side length
$\ell_m=2^{-cm}$, such that
\begin{equation}\label{eq:nilpotent-level-cubes}
    \begin{gathered}
        E_m\subseteq\bigcup_{j=1}^{J_m}Q_j^m
        \subseteq E_{m+1},\\
        J_m\ell_m^d\leq|E_{m+1}|,
        \qquad
        |(c_{Q_j^m})_i|\leq2^{Cm},
    \end{gathered}
\end{equation}
where $c,C>0$ depend only on the fixed coordinates
and $P$.
Moreover,
\begin{equation}\label{eq:nilpotent-kernel-domination}
    P^\beta
    \leq C_\beta
    \sum_{m\geq1}2^{-\beta m}\dsone_{E_m}.
\end{equation}

For fixed $m,j$, define
$$
A_{i,k}^{m,j}F(g)
=\frac1{\ell_m}
 \int_{(c_{Q_j^m})_i-\ell_m/2}^{(c_{Q_j^m})_i+\ell_m/2}
 F\bigl(g\exp(-e^{-k}sX_i)\bigr)\,ds.
$$
Apply \eqref{eq:transferred-interval-majorant}
and \eqref{eq:transferred-interval-norm}
with $b=(c_{Q_j^m})_i$ and $\ell=\ell_m$.
Since \eqref{eq:nilpotent-level-cubes} gives
$
\frac{|(c_{Q_j^m})_i|}{\ell_m}
\leq2^{(C+c)m},
$
we have
$$
\begin{aligned}
C_p\left(
    1+\log\left(2+\frac{|(c_{Q_j^m})_i|}{\ell_m}\right)
\right)^{1/p}
\leq C_{p,N,P}(1+m)^{1/p}.
\end{aligned}
$$

By the coordinate formula for the inverse, we get
\begin{equation}\label{eq:nilpotent-cube-average}
    \frac1{\ell_m^d}
    \int_{Q_j^m}
    F\bigl(g(\delta_{-h}n)^{-1}\bigr)\,dn
    =
    A_{d,h_d}^{m,j}\cdots A_{1,h_1}^{m,j}F(g).
\end{equation}
For $F\geq0$, set $U_0^{m,j}=F$.
By \eqref{eq:transferred-interval-majorant},
\eqref{eq:transferred-interval-norm} and the preceding
bound, we may successively choose
$U_i^{m,j}\in L_p(\cA)_+$ such that
$$
A_{i,k}^{m,j}U_{i-1}^{m,j}\leq U_i^{m,j}
\qquad(k\in\bZ),
$$
and
$$
\|U_i^{m,j}\|_p
\leq C_{p,N,P}(1+m)^{1/p}\|U_{i-1}^{m,j}\|_p.
$$
By positivity and \eqref{eq:nilpotent-cube-average},
\begin{equation}\label{eq:nilpotent-cube-majorant}
    \frac1{\ell_m^d}
    \int_{Q_j^m}
    F\bigl(g(\delta_{-h}n)^{-1}\bigr)\,dn
    \leq U_d^{m,j}(g)
    \qquad(h\in\bZ^d).
\end{equation}
Moreover, iterating the norm estimate yields
\begin{equation}\label{eq:nilpotent-cube-norm}
    \|U_d^{m,j}\|_p
    \leq C_{p,N,P}(1+m)^{d/p}\|F\|_p.
\end{equation}
Choose $\beta_0<\beta'<\beta$.
Chebyshev's inequality gives
$
|E_{m+1}|
\leq2^{\beta'(m+1)}\|P^{\beta'}\|_1.
$
Define
$$
U=C_\beta
\sum_{m\geq1}2^{-\beta m}\ell_m^d
\sum_{j=1}^{J_m}U_d^{m,j}.
$$
By \eqref{eq:nilpotent-cube-norm},
\eqref{eq:nilpotent-level-cubes} and the preceding
measure estimate,
$$
\begin{aligned}
\|U\|_p
&\leq
C_{p,N,P,\beta}
\sum_{m\geq1}
2^{-\beta m}(1+m)^{d/p}
J_m\ell_m^d\,\|F\|_p\\
&\leq
C_{p,N,P,\beta,\beta'}\|P^{\beta'}\|_1
\sum_{m\geq1}
(1+m)^{d/p}2^{-(\beta-\beta')m}\|F\|_p
<\infty.
\end{aligned}
$$
Thus the positive series defining $U$ converges
in $L_p(\cA)$.

Finally, \eqref{eq:nilpotent-kernel-domination},
\eqref{eq:nilpotent-level-cubes} and
\eqref{eq:nilpotent-cube-majorant} yield
$$
\begin{aligned}
T_hF(g)
\leq
C_\beta\sum_{m\geq1}2^{-\beta m}
\sum_{j=1}^{J_m}
\int_{Q_j^m}
F\bigl(g(\delta_{-h}n)^{-1}\bigr)\,dn\leq
U(g).
\end{aligned}
$$
Hence $U$ is a common positive majorant for
$(T_hF)_{h\in\bZ^d}$ with the required norm bound.

For general $F$, the positive operator matrices
$$
\begin{pmatrix}
    T_h(|F^*|)&T_hF\\
    (T_hF)^*&T_h(|F|)
\end{pmatrix}\geq0
$$
and Lemma~\ref{lemma:positive-block} give the
maximal estimate from the positive case, as in
the proof of Corollary~\ref{thm:main-hormander}.
\end{proof}

We next apply Proposition~\ref{prop:nilpotent-poisson}
to Poisson integrals on symmetric spaces.
We use the boundary coordinates and Poisson formulas
from \cite[Sections~2--5]{sjogren86admissible}.

Let $X=G/K$ be a symmetric space of noncompact type,
where $G$ is connected semisimple with finite centre
and $K$ is maximal compact.
Fix an Iwasawa decomposition $G=KAN$.
Write $\mathfrak a=\operatorname{Lie}(A)$, let
$\Sigma^+$ and $\Pi$ denote the positive and simple
restricted roots, respectively, and put
$$
2\rho
=\sum_{\alpha\in\Sigma^+}
 (\dim\mathfrak g_\alpha)\alpha,
$$
where $\mathfrak g_\alpha$ is the restricted root space.
Let $H:G\to\mathfrak a$ be the Iwasawa projection,
set $M_0=Z_K(A)$ and $B=M_0AN$, and equip
$S=G/B=K/M_0$ with its $K$-invariant probability
measure $\sigma$.
Write $\pi_B:G\to S$ for the quotient map.

Fix $E\subsetneq\Pi$ and define
$$
\mathfrak a(E)=\bigcap_{\alpha\in E}\ker\alpha,
\qquad
A(E)=\exp\mathfrak a(E),
$$
and
$$
A(E)_{\bZ}
=
\{a\in A(E):
  \alpha(\log a)\in\bZ
  \text{ for every }\alpha\in\Pi\setminus E\}.
$$
Let $\overline N(E)$ be the subgroup with Lie algebra
$$
\overline{\mathfrak n}(E)
=
\bigoplus_{\substack{\alpha\in\Sigma^+\\
                     \alpha|_{\mathfrak a(E)}\ne0}}
\mathfrak g_{-\alpha}.
$$
Let $G^E$ be the connected semisimple part of
$Z_G(A(E))$, and put $K^E=G^E\cap K$,
with normalized Haar measure $dk$.
The compact group $K^E$ normalizes $\overline N(E)$
and centralizes $A(E)$.

Normalize Haar measure $dn$ on $\overline N(E)$
so that the boundary-coordinate formula in
\cite[Section~3, p.~318]{sjogren86admissible} reads
\begin{equation}\label{eq:boundary-coordinate-measure}
    \int_S f(b)\,d\sigma(b)
    =
    \int_{K^E}\int_{\overline N(E)}
    f(\pi_B(nk))
    e^{-\langle2\rho,H(nk)\rangle}\,dn\,dk.
\end{equation}
With this normalization,
$$
\int_{\overline N(E)}
e^{-\langle2\rho,H(nk)\rangle}\,dn=1,
\qquad k\in K^E.
$$
The formula also holds for positive operator-valued
functions.
On the open boundary chart, the $\overline N(E)$
coordinate is unique, while the compact coordinate
is taken modulo $M_0\cap K^E$.

Choose the canonical basis above from the negative
root spaces, with $X_i\in\mathfrak g_{-\alpha_i}$,
and normalize it so that $dn=dx_1\cdots dx_d$.
The list $(\alpha_i)_{i=1}^d$ includes root
multiplicities.
For $a\in A(E)_{\bZ}$, put
$$
h(a)
=(\alpha_1(\log a),\ldots,\alpha_d(\log a))
\in\bZ^d,
\qquad
J(a)=e^{\langle2\rho,\log a\rangle}.
$$
Define $\kappa_a(n)=a^{-1}na$.
Then
\begin{equation}\label{eq:poisson-conjugation}
    \kappa_a(n)=\delta_{h(a)}n,
    \qquad
    d(\kappa_a n)
    =e^{\sum_i\alpha_i(\log a)}\,dn
    =J(a)\,dn.
\end{equation}
The last equality follows because the roots in the
subsystem generated by $E$ vanish on $\mathfrak a(E)$.

The scalar Poisson kernel
$$
P(n)=e^{-\langle2\rho,H(n^{-1})\rangle},
\qquad n\in\overline N(E),
$$
satisfies conditions \textup{(i)--(iii)} of
Proposition~\ref{prop:nilpotent-poisson}
with $\beta_0=1/2$; see
\cite[Section~5, p.~324]{sjogren86admissible}.

Put
$$
\cA_E
=L_\infty(\overline N(E),dn)\overline\otimes\cM,
\qquad
\cB_S
=L_\infty(S,\sigma)\overline\otimes\cM,
$$
with their product traces.
The Poisson integral of a boundary function is
$$
\mathcal P f(g)
=\int_K f(\pi_B(gk))\,dk,
\qquad g\in G,
$$
where $dk$ here denotes normalized Haar measure on $K$.

\begin{corollary}\label{cor:poisson-boundary}
Let $1<p<\infty$.
\begin{itemize}
    \item[\textup{(i)}]
    For every $\beta>1/2$, the operators
    $$
    Q_a^{(\beta)}F(g)
    =J(a)\int_{\overline N(E)}
    F(n)P\bigl(\kappa_a(n^{-1}g)\bigr)^\beta\,dn,
    \qquad a\in A(E)_{\bZ},
    $$
    satisfy
    $$
    \|(Q_a^{(\beta)}F)_{a\in A(E)_{\bZ}}
      \|_{L_p(\cA_E;\ell_\infty)}
    \leq C_{p,E,\beta}\|F\|_{L_p(\cA_E)}.
    $$

    \item[\textup{(ii)}]
    Let $L,U\subset\overline N(E)$ and
    $V\subset G^E$ be compact, and let
    $f\in L_p(\cB_S)$ be supported in $\pi_B(LK^E)$.
    Set
    $$
    \Pi_{a,n_0,v}f(g)
    =\mathcal P f(ga n_0v),
    \qquad
    \mathcal I=A(E)\times U\times V.
    $$
    Then
    $$
    \|(\Pi_{a,n_0,v}f)_{(a,n_0,v)\in\mathcal I}
      \|_{L_p(\cA_E;\ell_\infty(\mathcal I))}
    \leq C_{p,E,L,U,V}\|f\|_{L_p(\cB_S)}.
    $$
\end{itemize}
\end{corollary}

\begin{proof}
For \textup{(i)}, put $u=\kappa_a(n^{-1}g)$.
Using Unimodularity and \eqref{eq:poisson-conjugation},
$$
n=g(\delta_{-h(a)}u)^{-1},
\qquad
dn=J(a)^{-1}\,du.
$$
Consequently,
\begin{equation}\label{eq:poisson-identification}
    \begin{aligned}
        Q_a^{(\beta)}F(g)
        =
        \int_{\overline N(E)}
        F\bigl(g(\delta_{-h(a)}u)^{-1}\bigr)
        P(u)^\beta\,du=T_{h(a)}F(g).
    \end{aligned}
\end{equation}
Proposition~\ref{prop:nilpotent-poisson}
therefore proves \textup{(i)}.

For \textup{(ii)}, assume that $L$ is nonempty
and first let $f\geq0$.
Define
$$
F(n,k)=f(\pi_B(nk)),
\qquad
(n,k)\in\overline N(E)\times K^E,
$$
and put
$$
\widetilde{\cA}_E
=
L_\infty(\overline N(E)\times K^E,dn\,dk)
\overline\otimes\cM.
$$
Since the weight
$e^{-\langle2\rho,H(nk)\rangle}$ is continuous and
strictly positive on the compact set $L\times K^E$,
we have
$
c_L=\min_{(n,k)\in L\times K^E}
e^{-\langle2\rho,H(nk)\rangle}>0.
$
Applying \eqref{eq:boundary-coordinate-measure}
to $\tau(|f|^p)$ and using that $F(n,k)=0$
for $n\notin L$, we obtain
$$
\begin{aligned}
\|f\|_{L_p(\cB_S)}^p
&=
\int_{K^E}\int_L
\tau(|F(n,k)|^p)
e^{-\langle2\rho,H(nk)\rangle}\,dn\,dk\\
&\geq
c_L\int_{K^E}\int_L\tau(|F(n,k)|^p)\,dn\,dk.
\end{aligned}
$$
Taking $p$th roots, we get
\begin{equation}\label{eq:boundary-pullback-norm}
    \|F\|_{L_p(\widetilde{\cA}_E)}
    \leq c_L^{-1/p}\|f\|_{L_p(\cB_S)}.
\end{equation}

For $k\in K^E$, let $\kappa_k(n)=k^{-1}nk$.
It follows from the compactness of $K^E$ that
$\kappa_k$ preserves Haar measure $dn$.
Since $K^E$ centralizes $A(E)$, $\kappa_k$ commutes
with $\kappa_a$.
Thus
$\mathcal U F(g,k)=F(\kappa_k^{-1}(g),k)$
defines a trace-preserving automorphism of
$\widetilde{\cA}_E$.

Apply \textup{(i)}, with $\beta=1$, to the coefficient
algebra $L_\infty(K^E)\overline\otimes\cM$.
We get
$$
\begin{aligned}
R_aF(g,k)
=\mathcal U^{-1}Q_a^{(1)}\mathcal U F(g,k)=J(a)\int_{\overline N(E)}
F(n,k)P\bigl(\kappa_k\kappa_a(n^{-1}g)\bigr)\,dn,
\end{aligned}
$$
where $Q_a^{(1)}$ acts in the $\overline N(E)$ variable.
There is therefore a common positive majorant
$W\in L_p(\widetilde{\cA}_E)_+$ such that
$$
R_aF\leq W
\quad(a\in A(E)_{\bZ}),
\qquad
\|W\|_p\leq C_{p,E}\|F\|_p.
$$
Integration in the compact variable defines a
positive contraction
$$
\mathcal E:L_p(\widetilde{\cA}_E)\to L_p(\cA_E),
\qquad
\mathcal E W(g)=\int_{K^E}W(g,k)\,dk.
$$

To express the Poisson integral in these coordinates,
apply \eqref{eq:boundary-coordinate-measure} to
$b\mapsto f(ga\cdot b)$ and substitute
$n=gam a^{-1}$.
Since $a$ centralizes $K^E$ and belongs to $B$,
$$
\pi_B(gamk)=\pi_B(nk),
\qquad
m=\kappa_a(g^{-1}n),
\qquad
dm=J(a)\,dn.
$$
By the left $K$-invariance of $H$, we have
$$
\begin{aligned}
e^{-\langle2\rho,H(mk)\rangle}
=e^{-\langle2\rho,H(k^{-1}mk)\rangle}=P\bigl(\kappa_k\kappa_a(n^{-1}g)\bigr).
\end{aligned}
$$
Hence
\begin{equation}\label{eq:poisson-compact-average}
    \mathcal P f(ga)
    =\mathcal E R_aF(g)
    \leq\mathcal E W(g),
    \qquad a\in A(E)_{\bZ}.
\end{equation}
This is the operator-valued form of the compact
integration in
\cite[(4.7), p.~323]{sjogren86admissible}.

For an arbitrary $a\in A(E)$, define
$[a]\in A(E)_{\bZ}$ by
$$
\alpha(\log[a])
=\lfloor\alpha(\log a)\rfloor,
\qquad \alpha\in\Pi\setminus E.
$$
Then $[a]^{-1}a$ belongs to a fixed compact subset
of $A(E)$.
Consequently,
$
z=[a]^{-1}an_0v
$
belongs to a compact set $D$ depending only on
$E,U,V$.

By the scalar Poisson formula
\cite[(2.1), p.~316]{sjogren86admissible},
the pushforward of $\sigma$ under $b\mapsto z\cdot b$
has density
$$
\frac{d(z_*\sigma)}{d\sigma}(kM_0)
=e^{-\langle2\rho,H(z^{-1}k)\rangle}.
$$
Thus
$
C_D
=\max_{z\in D,\,k\in K}
e^{-\langle2\rho,H(z^{-1}k)\rangle}<\infty.
$
For $f\geq0$, it follows that
$$
\begin{aligned}
\mathcal P f(g[a]z)
=\int_S f(g[a]\cdot b)\,d(z_*\sigma)(b)\leq C_D\int_S f(g[a]\cdot b)\,d\sigma(b)=C_D\mathcal P f(g[a]).
\end{aligned}
$$
Combining this with \eqref{eq:poisson-compact-average},
we obtain
$$
\Pi_{a,n_0,v}f
\leq C_D\mathcal E W
\qquad((a,n_0,v)\in\mathcal I).
$$
Moreover, by \eqref{eq:boundary-pullback-norm},
$$
\|C_D\mathcal E W\|_p
\leq
C_DC_{p,E}c_L^{-1/p}\|f\|_{L_p(\cB_S)}.
$$
This proves the positive case.
The estimate for general $f$ follows by applying
the same argument to $|f^*|$ and $|f|$ and using
Lemma~\ref{lemma:positive-block}.
\end{proof}

\section{Applications II: Fourier multipliers, Schur multipliers,
and double operator integrals}
\label{sec:applications}

The applications in this section follow from a maximal
inequality for trace-preserving actions.
We first establish this inequality and then apply it
to Fourier multipliers, Schur multipliers and double
operator integrals.
For group von Neumann algebras, we also obtain bilateral
almost uniform convergence of dyadic Fourier summation
methods.

Throughout this section, unless otherwise stated, let
$(K_\lambda)_{\lambda\in\Lambda}\subset L_1(\bR^d)$
satisfy
\begin{equation}\label{eq:action-kernel-mass}
    \sup_{\lambda\in\Lambda}\|K_\lambda\|_1
    \leq A_1<\infty
\end{equation}
and the joint H\"ormander condition
\eqref{eq:hormander} with constant $C_H$.
We use the Fourier transform convention
$$
\widehat K_\lambda(\xi)
=\int_{\bR^d}K_\lambda(y)
  e^{-2\pi i\langle y,\xi\rangle}\,dy.
$$

Let $\alpha=(\alpha_y)_{y\in\bR^d}$ be a family of
trace-preserving $*$-automorphisms of $\cM$ satisfying
$\alpha_0=\mathrm{id}_{\cM}$ and
$\alpha_{y+z}=\alpha_y\alpha_z$.
Assume that $y\mapsto\alpha_y(x)$ is weak$^*$ continuous
for every $x\in\cM$, and define
\begin{equation}\label{actionconvolution}
    T_\lambda^\alpha x
    =\int_{\bR^d}K_\lambda(y)\alpha_y(x)\,dy.
\end{equation}

\begin{corollary}\label{thm:trans}
The family $(T_\lambda^\alpha)_{\lambda\in\Lambda}$
is of weak type $(1,1)$ and of strong type $(p,p)$
for $1<p\leq\infty$.
The constants depend only on $d,A_1,C_H$ and,
in the strong estimate, $p$.
\end{corollary}

\begin{proof}
Apply Corollary~\ref{thm:main-hormander} and
Lemma~\ref{lemma:positive-transference} to the kernels
$|K_\lambda|$.
The passage to complex-valued kernels and arbitrary
index sets follows as in the proof of
Corollary~\ref{thm:main-hormander}.
\end{proof}

\subsection{Fourier multipliers}

We apply Corollary~\ref{thm:trans} to the natural
trace-preserving actions on group von Neumann algebras,
quantum Euclidean spaces and quantum tori.

\subsubsection{Group von Neumann algebras}
\label{subsec:group-multipliers}

Let $G$ be a countable discrete group, and let
$\mathrm{VN}(G)$ be its group von Neumann algebra,
generated by the left regular unitaries
$(u_g)_{g\in G}$.
Its canonical trace is given by
$\tau_G(u_g)=\dsone_{\{g=e\}}$.
We consider Fourier multipliers whose symbols have
the form $\widehat K_\lambda\circ b$, where
$b:G\to\bR^d$ is a group homomorphism.
For a Fourier polynomial $x=\sum_g x_gu_g$, put
$$
T_\lambda^b x
=\sum_g\widehat K_\lambda(b(g))x_gu_g.
$$

\begin{corollary}\label{cor:group-multipliers}
The family $(T_\lambda^b)_{\lambda\in\Lambda}$ extends
to a family of weak type $(1,1)$ and of strong type
$(p,p)$ for $1<p\leq\infty$ on the noncommutative
$L_p$-spaces associated with $\mathrm{VN}(G)$.
The constants are independent of $G$ and $b$.
\end{corollary}

\begin{proof}
Since $b$ is a group homomorphism,
$
\alpha_y(u_g)
=e^{-2\pi i\langle y,b(g)\rangle}u_g
$
defines a weak$^*$ continuous trace-preserving action
of $\bR^d$ on $\mathrm{VN}(G)$.
Indeed, $\alpha_y(x)=v_yxv_y^*$, where
$$
(v_y\xi)(g)
=e^{-2\pi i\langle y,b(g)\rangle}\xi(g),
\qquad \xi\in\ell_2(G).
$$
For a Fourier polynomial $x$, we have
$$
T_\lambda^b x
=\int_{\bR^d}K_\lambda(y)\alpha_y(x)\,dy.
$$
The conclusion follows from
Corollary~\ref{thm:trans}.
\end{proof}

We next apply this estimate to Fourier summation.
Let $K\in L_1(\bR^d)$ satisfy \eqref{eq:L1 Dini}
and \eqref{eq:size condition}, with $\int K=1$.
For $n\geq0$, define
$$
T_n^b x
=\sum_g\widehat K(2^{-n}b(g))x_gu_g
$$
initially on Fourier polynomials.

\begin{corollary}\label{cor:group-summation}
The family $(T_n^b)_{n\geq0}$ is of weak type $(1,1)$
and of strong type $(p,p)$ for $1<p\leq\infty$.
Moreover, for every $x\in L_1(\mathrm{VN}(G))$,
$$
T_n^b x\longrightarrow x
\quad\text{in }L_1(\mathrm{VN}(G))
\text{ and b.a.u.}
$$
\end{corollary}

\begin{proof}
Proposition~\ref{prop:dini-hormander} and
Corollary~\ref{cor:group-multipliers} give the
maximal estimates.
Since $\widehat K$ is continuous and
$\widehat K(0)=1$, every Fourier polynomial satisfies
$$
\|T_n^b x-x\|_\infty
\leq
\sum_g|x_g|
\bigl|\widehat K(2^{-n}b(g))-1\bigr|
\longrightarrow0.
$$
The weak type estimate gives bilateral uniform
equicontinuity at zero.
By the noncommutative Banach principle
\cite[Theorem~3.4]{chilinlitvinov06equicontinuity}
and density of Fourier polynomials in
$L_1(\mathrm{VN}(G))$, the sequence $(T_n^b x)_n$
converges b.a.u. for every $x\in L_1(\mathrm{VN}(G))$.

On the other hand,
$$
T_n^b x
=\int_{\bR^d}K(y)\alpha_{2^{-n}y}(x)\,dy.
$$
By strong continuity on $L_1$ and dominated convergence theorem, we have
$$
\|T_n^b x-x\|_1
\leq
\int_{\bR^d}|K(y)|
\|\alpha_{2^{-n}y}(x)-x\|_1\,dy
\longrightarrow0.
$$
\end{proof}

\subsubsection{Quantum Euclidean spaces}
\label{sec:euclidean}

Let $\theta$ be a real antisymmetric $d\times d$
matrix. The quantum Euclidean space
$\mathcal R_\theta^d$ is the von Neumann algebra
generated on $L_2(\bR^d)$ by the unitaries
$$
(U(t)\xi)(r)
=e^{-\frac i2\langle t,\theta r\rangle}\xi(r-t),
\qquad t\in\bR^d.
$$
For $g\in L_1(\bR^d)$, write
$U(g)=\int_{\bR^d}g(t)U(t)\,dt$.
We use the canonical normal faithful semifinite trace,
normalized by
$$
\tau_{\mathcal R}(U(g))=(2\pi)^dg(0),
\qquad g\in\mathcal S(\bR^d).
$$
The operators $U(g)$, $g\in\mathcal S(\bR^d)$,
form the Schwartz class and are dense in
$L_p(\mathcal R_\theta^d)$ for $1\leq p<\infty$;
see~\cite{gonzalezjungeparcet19singular,
mcdonaldsukochevxiong20quantum}.

The canonical action
$$
\alpha_y(U(t))
=e^{-2\pi i\langle y,t\rangle}U(t)
$$
is weak$^*$ continuous and trace-preserving.
It is given by conjugation with the multiplication
operator $\xi(r)\mapsto e^{-2\pi i\langle y,r\rangle}\xi(r)$
on $L_2(\bR^d)$.
Define the Fourier multipliers on the Schwartz class by
$$
T_\lambda U(g)=U(\widehat K_\lambda g).
$$

\begin{corollary}\label{cor:euclid}
The family $(T_\lambda)_{\lambda\in\Lambda}$ extends
to a family of weak type $(1,1)$ and of strong type
$(p,p)$ for $1<p\leq\infty$ on
$L_p(\mathcal R_\theta^d)$.
\end{corollary}

\begin{proof}
For $g\in\mathcal S(\bR^d)$, we have
$$
\begin{aligned}
T_\lambda U(g)
=\int_{\bR^d}K_\lambda(y)
  U(e^{-2\pi i\langle y,\cdot\rangle}g)\,dy=\int_{\bR^d}K_\lambda(y)\alpha_y(U(g))\,dy.
\end{aligned}
$$
Then, the conclusion follows from
Corollary~\ref{thm:trans}.
\end{proof}

\subsubsection{Quantum tori}
\label{sec:torus}

Let $\theta$ be a real antisymmetric $d\times d$
matrix. Let $\mathcal A_\theta$ be the universal
$C^*$-algebra generated by unitaries $U_1,\ldots,U_d$
satisfying
$$
U_kU_j=e^{2\pi i\theta_{kj}}U_jU_k.
$$
Its canonical trace is determined by
$$
\tau_\theta(U^m)=\dsone_{\{m=0\}},
\qquad
U^m=U_1^{m_1}\cdots U_d^{m_d},
\qquad m\in\bZ^d.
$$
Write $\mathbb T_\theta^d$ for the von Neumann algebra
obtained in the GNS representation of this trace.
The canonical action
$$
\alpha_y(U^m)
=e^{-2\pi i\langle y,m\rangle}U^m,
\qquad y\in\bR^d,
$$
is weak$^*$ continuous and trace-preserving;
see~\cite{chenxuyin13harmonic}.
For a trigonometric polynomial
$x=\sum_m\widehat x(m)U^m$, define
$$
T_\lambda x
=\sum_m\widehat K_\lambda(m)\widehat x(m)U^m.
$$

\begin{corollary}\label{cor:tori}
The family $(T_\lambda)_{\lambda\in\Lambda}$ extends
to a family of weak type $(1,1)$ and of strong type
$(p,p)$ for $1<p\leq\infty$ on
$L_p(\mathbb T_\theta^d)$.

\end{corollary}

\begin{proof}
For trigonometric polynomials, we have
$$
\begin{aligned}
T_\lambda x
=\sum_m\widehat x(m)
  \left(
      \int_{\bR^d}K_\lambda(y)
      e^{-2\pi i\langle y,m\rangle}\,dy
  \right)U^m=\int_{\bR^d}K_\lambda(y)\alpha_y(x)\,dy.
\end{aligned}
$$
Then, the conclusion follows from
Corollary~\ref{thm:trans}.
\end{proof}

\subsection{Schur multipliers}

Let $G$ be a locally compact group with left Haar
measure $\mu$, and let $\beta:G\to\bR^d$ be continuous.
For $y\in\bR^d$, define a unitary on $L_2(G)$ by
$$
(u_y\xi)(g)
=e^{-2\pi i\langle y,\beta(g^{-1})\rangle}\xi(g).
$$
Then $\alpha_y(a)=u_yau_y^*$ is a weak$^*$ continuous,
trace-preserving action on $\mathcal B(L_2(G))$,
equipped with its canonical trace.

For a Hilbert--Schmidt operator $a$ with integral
kernel $a(g,h)$, define the Schur multiplier
$$
(S_\lambda a)(g,h)
=
\widehat K_\lambda
\bigl(\beta(g^{-1})-\beta(h^{-1})\bigr)a(g,h).
$$

\begin{corollary}\label{schue}
The family $(S_\lambda)_{\lambda\in\Lambda}$ extends
to a family of weak type $(1,1)$ on $S_1(L_2(G))$
and of strong type $(p,p)$ on $S_p(L_2(G))$
for $1<p\leq\infty$, where
$S_\infty(L_2(G))=\mathcal B(L_2(G))$.
\end{corollary}

\begin{proof}
Since the integral kernel of $u_yau_y^*$ is
$
e^{-2\pi i\langle y,\beta(g^{-1})-\beta(h^{-1})\rangle}
a(g,h),
$
we have
$$
S_\lambda a
=\int_{\bR^d}K_\lambda(y)u_yau_y^*\,dy
$$
on the Hilbert--Schmidt class.
Then, the conclusion follows from
Corollary~\ref{thm:trans}.
\end{proof}

In particular, the result applies when $\beta$ is
a continuous orthogonal cocycle.
Such cocycles also occur in the study of Schur
multipliers in
\cite{cadilhaccondeparcet22spectral}.

\subsection{Double operator integrals}
\label{subsec:double-operator-integrals}

Double operator integrals are a standard tool in
operator perturbation theory and provide a spectral
formulation of Schur multipliers.
For background, see
\cite[Section~1.1]{peller16multiple}
and, in the semifinite setting,
\cite{depagterwitvlietsukochev02double}.
Here we obtain maximal inequalities for families
whose symbols depend on the difference of the
spectral variables.

In this subsection, take $d=1$.
Let $H$ be a selfadjoint operator affiliated with
$\cM$, with spectral measure $E_H$.
For $x\in L_2(\cM)$, define
$$
 T_\lambda^H x
=\int_{\bR}\int_{\bR}
 \widehat K_\lambda(s-t)\,dE_H(s)\,x\,dE_H(t).
$$
Therefore, one has
\begin{equation}\label{eq:doi-integrated-action}
     T_\lambda^H x
    =\int_{\bR}K_\lambda(y)
      e^{-2\pi iyH}x e^{2\pi iyH}\,dy.
\end{equation}
For every $1\leq p\leq\infty$, this formula defines
a bounded operator on $L_p(\cM)$ satisfying
$$
\| T_\lambda^H x\|_p
\leq\|K_\lambda\|_1\|x\|_p.
$$

\begin{corollary}\label{cor:doi-maximal}
The family
$( T_\lambda^H)_{\lambda\in\Lambda}$
is of weak type $(1,1)$ and of strong type $(p,p)$
for $1<p\leq\infty$.
\end{corollary}

\begin{proof}
The maps
$\alpha_y(x)=e^{-2\pi iyH}xe^{2\pi iyH}$
form a weak$^*$ continuous trace-preserving action
of $\bR$ on $\cM$.
Then, the conclusion follows from
\eqref{eq:doi-integrated-action} and
Corollary~\ref{thm:trans}.
\end{proof}

In particular, for a diagonal matrix
$H=\operatorname{diag}(h_1,\ldots,h_N)$,
$$
( T_\lambda^H x)_{ij}
=\widehat K_\lambda(h_i-h_j)x_{ij}.
$$
Thus the maximal estimates are uniform in $N$
and in the choice of $h_1,\ldots,h_N$.

\bigskip
\noindent\textbf{Acknowledgements.}
G. Hong was partially supported by the National Natural Science
Foundation of China (Grant Nos.~12325105, W2441002, and W2611005).
L. Wang was supported by a grant from the Research Grants Council of the
Hong Kong Administrative Region, China (Project No.~CityU 11308924).
X. Wang was supported by the National Natural Science Foundation of China
(Grant Nos.~12601242 and W2611005).
B. Xu was supported by the National Natural Science Foundation of China
(Grant No.~12501172) and the Fundamental Research Funds for the Central
Universities (Grant No.~20720250061).

\bigskip

\noindent\textbf{AI statement.}
The authors used GPT-6-Astra to assist with their
research on this topic. The model also assisted with language polishing
and editing. The authors checked and developed the suggested idea,
verified the proofs, and take full responsibility for the content of the paper.

\bibliographystyle{amsalpha}
\bibliography{main}

\end{document}